\documentclass[11pt,reqno]{article}

\usepackage{amsmath, amssymb, amsthm} 
\usepackage{empheq}
\usepackage{mathpazo}
\usepackage{nameref}
\usepackage[margin=2.5cm]{geometry}
\usepackage[titletoc]{appendix}

\usepackage{caption,subcaption}
\usepackage[doc]{optional}
\usepackage{graphicx, xcolor, soul}
\usepackage{float}
\restylefloat{table*}
\usepackage{booktabs} 
\definecolor{labelkey}{rgb}{0,0.08,0.45}
\definecolor{refkey}{rgb}{0,0.6,0.0}
\definecolor{Brown}{rgb}{0.45,0.0,0.05}
\definecolor{lime}{rgb}{0.00,0.8,0.0}
\definecolor{lblue}{rgb}{0.5,0.5,0.99}

\usepackage{hyperref}

\usepackage{sidecap}
\usepackage[shortlabels]{enumitem}
\setlist[enumerate]{nosep}

\definecolor{myblue}{rgb}{.9, .9, 1} 
  \newcommand*\mybluebox[1]{%
    \colorbox{myblue}{\hspace{1em}#1\hspace{1em}}}

\newcommand{\sepp}{\setlength{\itemsep}{-2pt}}

\newcommand{\menge}[2]{\{{#1}~\big |~{#2}\}}

\newcommand{\ball}[2]{\operatorname{ball}\left({#1};{#2}\right)}

\newcommand{\scal}[2]{\left\langle {#1},{#2} \right\rangle}

\newcommand{\To}{\ensuremath{\rightrightarrows}}

\newcommand{\NN}{\ensuremath{\mathbb N}}
\newcommand{\nnn}{\ensuremath{{n\in{\mathbb N}}}}

\newcommand{\RR}{\ensuremath{\mathbb R}}
\newcommand{\RP}{\ensuremath{\mathbb{R}_+}}
\newcommand{\RPP}{\ensuremath{\mathbb{R}_{++}}}

\newcommand{\RM}{\ensuremath{\mathbb{R}_-}}

\newcommand{\argmin}{\ensuremath{\operatorname*{argmin}}}

\newcommand{\bd}{\ensuremath{\operatorname{bdry}}}
\newcommand{\inte}{\ensuremath{\operatorname{int}}}

\newcommand{\reli}{\ensuremath{\operatorname{ri}}}

\newcommand{\cone}{\ensuremath{\operatorname{cone}}}
\newcommand{\aff}{\ensuremath{\operatorname{aff}}}
\newcommand{\lspan}{\ensuremath{{\operatorname{span}}\,}}

\newcommand{\ran}{\ensuremath{\operatorname{ran}}}

\newcommand{\epi}{\ensuremath{\operatorname{epi}}}

\newcommand{\Fix}{\ensuremath{\operatorname{Fix}}}
\newcommand{\Id}{\ensuremath{\operatorname{Id}}}

\usepackage[capitalize]{cleveref}
\crefname{equation}{}{equations}
\crefname{chapter}{Appendix}{chapters}
\crefname{item}{}{items}
\crefname{figure}{Figure}{figures}
\makeatletter
\def\namedlabel#1#2{\begingroup
   \def\@currentlabel{#2}%
   \label{#1}\endgroup
}
\makeatother

\makeatletter
\def\th@plain{%
  \thm@notefont{}% same as heading font
  \itshape % body font
}
\def\th@definition{%
  \thm@notefont{}% same as heading font
  \normalfont % body font
}
\makeatother

\newtheorem{theorem}{Theorem}[section]
\newtheorem{lemma}[theorem]{Lemma}
\newtheorem{corollary}[theorem]{Corollary}
\newtheorem{proposition}[theorem]{Proposition}
\newtheorem{fact}[theorem]{Fact}
\theoremstyle{definition}
\newtheorem{definition}[theorem]{Definition}
\theoremstyle{definition}
\newtheorem{example}[theorem]{Example}
\theoremstyle{definition}
\newtheorem{remark}[theorem]{Remark}

\allowdisplaybreaks

\begin{document}

\title{On the finite convergence of the Douglas--Rachford algorithm
for solving (not necessarily convex) feasibility problems in
Euclidean spaces}

\author{
Heinz H.\ Bauschke\thanks{
Mathematics, University of British Columbia, Kelowna, B.C.\ V1V~1V7, Canada. 
E-mail: \texttt{heinz.bauschke@ubc.ca}.}~~~and~
Minh N.\ Dao\thanks{
Mathematics, University of British Columbia, Kelowna, B.C.\ V1V~1V7, Canada,
and Department of Mathematics and Informatics, Hanoi National University of Education, 136 Xuan Thuy, Hanoi, Vietnam.
E-mail: \texttt{minhdn@hnue.edu.vn}.}
}

\date{April 15, 2016}

\maketitle

\begin{abstract} \noindent
Solving feasibility problems is a central
task in mathematics and the applied sciences.
One particularly successful method is the Douglas--Rachford
algorithm. 
In this paper, we provide many new conditions
sufficient for \emph{finite} convergence. 
Numerous examples illustrate our results. 
\end{abstract}

{\small

}

\section{Introduction}

The Douglas--Rachford algorithm (DRA) was first introduced in \cite{DR56} 
as an operator splitting technique to solve partial differential equations arising in heat conduction. 
As a result of findings by Lions and Mercier \cite{LM79} in the monotone operator
setting, the method has been extended to find solutions of the
sum of two maximally monotone operators. 
When specialized to normal cone operators, the method is very
useful in solving feasibility problems. To fix our setting, 
we assume throughout that
\begin{empheq}[box=\mybluebox]{equation}
\text{$X$ is a Euclidean space,}
\end{empheq}
i.e, a finite-dimensional real Hilbert space with inner product $\scal{\cdot}{\cdot}$ and induced norm $\|\cdot\|$.
Given closed subsets $A$ and $B$ of $X$ with nonempty intersection, 
we consider the fundamental feasibility problem  
\begin{empheq}[box=\mybluebox]{equation}
\label{e:prob}
\text{find a point in $A\cap B$}
\end{empheq}
which 
frequently arises in science and engineering applications.
A common approach for solving \eqref{e:prob} is to use projection algorithms
that employ projectors onto the underlying sets; see, e.g., 
\cite{BB96}
\cite{BC11},
\cite{Cegielski}, 
\cite{CenZen},
\cite{C3}, 
\cite{Comb95}, 
\cite{Comb97},
\cite{Herman},
and the references therein.
Among those algorithms, the Douglas--Rachford algorithm 
applied to \eqref{e:prob} 
has attracted much attention; 
see, e.g., \cite{ABT14} and \cite{Comb04} 
and the references therein for further information.

In the convex case, it is known, see, e.g., Lions and Mercier \cite{LM79} and Svaiter \cite{Sva11}, that  
the sequence generated by the DRA always converges 
while the ``shadow sequence'' converges to a point of the intersection.
Even when the convex feasibility problem is inconsistent, i.e., $A\cap B =\varnothing$,
it was shown in \cite{BCL04} that the ``shadow sequence'' is bounded
and its cluster points solve a best approximation problem;
the entire sequence converges if one of the sets is an affine subspace
\cite{BDM16}.

Although the Douglas--Rachford algorithm has been applied successfully to various problems involving
one or more nonconvex sets, the theoretical justification is far from complete. 
Recently, in the case of a Euclidean sphere and a line, Borwein
and Sims
\cite{BS11} have proved local convergence of the DRA at points of the intersection,  
while Arag\'on Artacho and Borwein
\cite{AB13} have given a region of convergence for this model in
the plane; moreover, 
Benoist \cite{Ben15} has even shown that the DRA sequence converges in norm to a point of the intersection 
except when the starting point belongs to the hyperplane of symmetry.
In another direction, \cite{BN14} proved local convergence for finite unions of convex sets.

On the convergence rate, it has been shown by Hesse, Luke and
Neumann \cite{HLN14} that the DRA for two
subspaces converges linearly.
Furthermore, the rate is then actually the cosine of the Friedrichs angle
between the subspaces \cite{BBNPW14}.  In the potentially nonconvex
case, under transversality assumptions,
Hesse and Luke \cite{HL13} proved local linear convergence of the DRA for a superregular set and an affine subspace,
while Phan \cite{Pha16} obtained such a rate for two super-regular sets.
Specialized to the convex setting, the result in \cite{Pha16} implies linear convergence of the DRA 
for two convex sets whose the relative interiors have a nonempty intersection; see also \cite{BNP15}. 
It is worth mentioning that the linear convergence of the DRA may fail 
even for simple settings in the Euclidean plane, as shown in \cite{BDNP16}. 
Based on H\"older regularity properties, Borwein, Li, and Tam 
\cite{BLT15} established sublinear convergence 
for two convex basic semi-algebraic sets. 
For the linear convergence of the DRA in the framework of optimization problems involving a sum of two functions,
we refer the reader to, e.g., Giselsson's \cite{Gis15a}, \cite{Gis15b}, 
Li and Pong \cite{LP15}, Liang, Faili, Peyr\'e, and Luke \cite{LPFL15}, 
Patrinos, Stella, and Bemporad's \cite{PSB14}, and the references therein.

Davis and Yin \cite{DY15} observed that the DRA may converge arbitrarily
slowly in infinite dimensions; however, in finite
dimensions, 
it often works extremely well. 
Very recently, the globally finite convergence of the DRA has been shown 
in \cite{BDNP15} for an affine subspace and a locally polyhedral set, or for a hyperplane and an epigraph, 
and then by Arag\'on Artacho, Borwein, and Tam \cite{ABT15} 
for a finite set and a halfspace.

\emph{The goal of this paper is to provide various finite-convergence results.}
The sufficient conditions we present are new
and complementary to existing conditions.

After presenting useful results on projectors and the DRA
(Section~\ref{s:aux})
and on locally identical sets (Section~\ref{s:lis}), we specifically derive 
results related to the following five scenarios:
\begin{itemize} 
\sepp
\item[\bf R1]
$A$ is a halfspace and $B$ is an epigraph of a convex function;
$A$ is either a hyperplane or a halfspace, and $B$ is a halfspace
(see Section~\ref{s:4}).
\item[\bf R2]
$A$ and $B$ are supersets or modifications of other sets where the DRA is better
understood 
(see Section~\ref{s:expand}).
\item[\bf R3]
$A$ and $B$ are subsets of other sets where the DRA is better
understood 
(see Section~\ref{s:shrink}).
\item[\bf R4]
$B$ is a finite, hence nonconvex, set
(see Section~\ref{s:finiteset}). 
\item[\bf R5]
$A$ is an affine subspace and $B$ is a polyhedron
in the absence of Slater's condition
(see Section~\ref{s:8}).
\end{itemize}
The paper concludes with a list of open problem in
Section~\ref{s:open}. 

Before we start our analysis, let us note that our 
notation and terminology is standard and follows, e.g., \cite{BC11}.
The nonnegative integers are $\NN$, and the real numbers are $\RR$, 
while $\RP := \menge{\alpha \in \RR}{\alpha \geq 0}$, $\RPP := \menge{\alpha \in \RR}{\alpha >0}$, 
and $\RM := \menge{\alpha \in \RR}{\alpha \leq 0}$.
Let $C$ be a subset of $X$. 
Then the closure of $C$ is $\overline C$, the interior of $C$ is $\inte C$, the boundary of $C$ is $\bd C$,
and the smallest affine and linear subspaces containing $C$ are,
respectively, $\aff C$ and $\lspan C$. 
The relative interior of $C$, $\reli C$, is the interior of $C$ relative to $\aff C$.
The smallest convex cone containing $C$ is $\cone C$,
the orthogonal complement of $C$ is $C^\perp :=\menge{y \in X}{(\forall x\in C)\; \scal{x}{y} =0}$,
and the dual cone of $C$ is $C^\oplus :=\menge{y \in X}{(\forall x \in C)\; \scal{x}{y} \geq 0}$.
The normal cone operator of $C$ is denoted by $N_C$, i.e.,
$N_C(x) =\menge{y \in X}{(\forall c \in C)\; \scal{y}{c-x} \leq 0}$ if $x \in C$, and $N_C(x) =\varnothing$ otherwise.
If $x \in X$ and $\rho \in \RPP$, 
then $\ball{x}{\rho} :=\menge{y \in X}{\|x -y\| \leq \rho}$ 
is the closed ball centered at $x$ with radius $\rho$.

\section{Auxiliary results}

\label{s:aux}

For the reader's convenience we recall in this section preliminary concepts and auxiliary results 
which are mostly well known and which will be useful later.

Let $A$ be a nonempty closed subset of $X$. The \emph{distance function} of $A$ is 
\begin{equation}
d_A\colon X \to \RR\colon x \mapsto \min_{a \in A} \|x -a\|.
\end{equation}
The \emph{projector} onto $A$ is the mapping
\begin{equation}
P_A\colon X \To A\colon x \mapsto \argmin_{a \in A} \|x -a\| =\menge{a \in A}{\|x -a\| =d_A(x)},
\end{equation}
and the \emph{reflector} across $A$ is defined by 
\begin{equation}
R_A :=2P_A -\Id,
\end{equation}
where $\Id$ is the identity operator. 
Note that closedness of the set $A$ is necessary and sufficient for $A$ to be proximinal, i.e., $(\forall x \in X)$ $P_Ax \neq\varnothing$ (see, e.g., \cite[Corollary~3.13]{BC11}). In the following, we shall write $P_Ax =a$ if $P_Ax =\{a\}$ is a singleton.

\begin{fact}[Projection onto a convex set]
\label{f:proj}
Let $A$ be a nonempty closed convex subset of $X$, 
and let $x$ and $p$ be in $X$. Then the following hold:
\begin{enumerate}
\item 
\label{f:proj_cvx} 
$P_A$ is single-valued and 
\begin{equation}
\label{e:proj_cvx}
p =P_Ax \quad\Leftrightarrow\quad \left[ p\in A \text{~and~} (\forall y \in A)\; \scal{x -p}{y -p}\leq 0 \right]
\quad\Leftrightarrow\quad x -p \in N_A(p).
\end{equation}
\item
\label{f:proj_firm}
$P_A$ is \emph{firmly nonexpansive}, i.e., 
\begin{equation}
(\forall x \in X)(\forall y \in X)\quad \|P_Ax -P_Ay\|^2 +\|(\Id -P_A)x -(\Id -P_A)y\|^2 \leq \|x -y\|^2.
\end{equation}
\item
\label{f:proj_nonex}
$R_A$ is \emph{nonexpansive}, i.e., 
\begin{equation}
(\forall x \in X)(\forall y \in X)\quad \|R_Ax -R_Ay\| \leq \|x -y\|.
\end{equation}
\end{enumerate} 
In particular, $P_A$ and $R_A$ are continuous on $X$.
\end{fact}
\begin{proof}
\ref{f:proj_cvx}: \cite[Theorem~3.14 and Proposition~6.46]{BC11}. 
\ref{f:proj_firm}: \cite[Proposition~4.8]{BC11}.
\ref{f:proj_nonex}: \cite[Corollary~4.10]{BC11}.
\end{proof}

\begin{lemma}
\label{l:projAB}
Let $A$ and $B$ be closed subsets of $X$ such that $A \subseteq B$, and let $x \in X$.
Then the following hold:
\begin{enumerate}
\item 
\label{l:projAB_sub}
$A\cap P_Bx \subseteq P_Ax$. 
\item 
\label{l:projAB_preimage}
$(\forall p \in A)$ $P_B^{-1}p \subseteq P_A^{-1}p$. 
\item
\label{l:projAB_single}
If $P_Bx =p \in A$, then $P_Ax =P_Bx$. 
\item
\label{l:projAB_cvx}
If $B$ is convex and $P_Bx \in A$, then $P_Ax =P_Bx$.
\end{enumerate}
\end{lemma}
\begin{proof}
\ref{l:projAB_sub}: The conclusion is obvious if $A\cap P_Bx =\varnothing$.
Assume $A\cap P_Bx \neq\varnothing$, and let $p \in A\cap P_Bx$.
Then $\|x -p\| \leq \|x -y\|$ for all $y \in B$, and so for all $y \in A$ since $A \subseteq B$.
This combined with $p \in A$ gives $p \in \argmin_{y \in A} \|x -y\| =P_Ax$.

\ref{l:projAB_preimage}: Let $p \in A$. For all $x \in P_B^{-1}p$, we have $p \in P_Bx$,
and by \ref{l:projAB_sub}, $p \in A\cap P_Bx \subseteq P_Ax$, which implies $x \in P_A^{-1}p$.

\ref{l:projAB_single}: Assume that $P_Bx =p \in A$. 
Using \ref{l:projAB_sub}, we have $p \in P_Ax$, and so
\begin{equation}
P_Ax =\menge{y \in A}{\|x -y\| =\|x -p\|} \subseteq \menge{y \in B}{\|x -y\| =\|x -p\|} =P_Bx =\{p\}.
\end{equation}
It follows that $P_Ax =P_Bx =\{p\}$.

\ref{l:projAB_cvx}: By \cref{f:proj}\ref{f:proj_cvx}, if $B$ is convex, then $P_Bx$ is a singleton,
and if additionally $P_Bx \in A$, then by \ref{l:projAB_single}, $P_Ax =P_Bx$. 
\end{proof}

\begin{example}[Projection onto an affine subspace]
\label{ex:proj_affine}
Let $Y$ be a real Hilbert space, let $L$ be a linear operator from $X$ to $Y$, 
let $v \in \ran L$, and set $A =\menge{x \in X}{Lx =v}$.
Then 
\begin{equation}
(\forall x \in X)\quad P_Ax =x -L^\dagger(Lx -v),
\end{equation}
where $L^\dagger$ denotes the Moore--Penrose inverse of $L$.
\end{example}
\begin{proof}
This follows from \cite[Lemma~4.1]{BK04}, see also \cite[Example~28.14]{BC11}.
\end{proof}

\begin{example}[Projection onto a hyperplane or a halfspace]
\label{ex:proj_h}
Let $u \in X\smallsetminus \{0\}$, and let $\eta \in \RR$.
Then the following hold:
\begin{enumerate}
\item
\label{ex:proj_hyper} 
If $A =\menge{x \in X}{\scal{x}{u} =\eta}$, then
\begin{equation}
(\forall x \in X)\quad P_Ax =x -\frac{\scal{x}{u} -\eta}{\|u\|^2}u. 
\end{equation}
\item
\label{ex:proj_half}
If $A =\menge{x \in X}{\scal{x}{u} \leq \eta}$, then
\begin{equation}
(\forall x \in X)\quad P_Ax =\begin{cases}
x &\text{~if~} \scal{x}{u} \leq \eta, \\ 
x -\frac{\scal{x}{u} -\eta}{\|u\|^2}u &\text{~if~} \scal{x}{u} > \eta. 
\end{cases} 
\end{equation}
\end{enumerate} 
\end{example}
\begin{proof}
\ref{ex:proj_hyper}: \cite[Example~28.15]{BC11}. \ref{ex:proj_half}: \cite[Example~28.16]{BC11}.
\end{proof}

\begin{example}[Projection onto a ball]
\label{ex:proj_ball}
Let $B =\ball{u}{\rho}$ with $u \in X$ and $\rho \in \RPP$. Then
\begin{equation}
(\forall x \in X)\quad P_Bx =u +\frac{\rho}{\max\{\|x -u\|, \rho\}}(x -u).
\end{equation}
\end{example}
\begin{proof}
Let $x \in X$.
We have to prove $P_Bx =x$ if $\|x -u\| \leq \rho$, and $P_Bx =b :=u +\frac{\rho}{\|x -u\|}(x -u)$ otherwise.
Indeed, if $\|x -u\| \leq \rho$, then $x \in B$, and thus $P_Bx =x$. 
Assume that $\|x -u\| >\rho$.  
On the one hand, for all $y \in B$, by using $\|y -u\| \leq \rho$ and the triangle inequality, 
\begin{equation}
\|x -b\| =\|x -u\| -\rho \leq \|x -u\| -\|y -u\| \leq \|x -y\|.
\end{equation}
On the other hand, $\|b -u\| =\rho$, and so $b \in \ball{u}{\rho}$,  
then by combining with the convexity of $B$ and the above inequality, $P_Bx =b$, which completes the formula.
\end{proof}

\begin{example}[Projection onto an epigraph]
\label{ex:proj_epi}
Let $f \colon X \to \RR$ be convex and continuous, set $B =\epi f :=\menge{(x, \rho) \in X\times \RR}{f(x)\leq\rho}$, 
and let $(x, \rho) \in (X\times \RR)\smallsetminus B$. 
Then there exists $p\in X$ such that $P_B(x, \rho) =(p, f(p))$, 
\begin{equation}
x \in p +\big(f(p) -\rho\big)\partial f(p)\;\text{and}\; \rho <f(p) \leq
f(x)
\end{equation}
and 
\begin{equation}
\label{e:epi_proj}
(\forall y \in X)\quad \scal{y -p}{x -p} \leq \big(f(y)
-f(p)\big)\big(f(p) -\rho\big).
\end{equation}
\end{example}
\begin{proof}
See \cite[Lemma~5.1]{BDNP15}.
\end{proof}

In order to solve the feasibility problem \eqref{e:prob}, where $A$ and $B$ are closed subsets of $X$ with nonempty intersection,  
we employ the \emph{Douglas--Rachford algorithm} (also called \emph{averaged alternating reflections}) that generates a sequence $(x_n)_\nnn$ 
by
\begin{empheq}[box=\mybluebox]{equation}
\label{e:DRAseq}
(\forall\nnn)\quad x_{n+1} \in T_{A,B}x_n,\quad \text{where~} x_0 \in X,
\end{empheq}
and where 
\begin{empheq}[box=\mybluebox]{equation}
T_{A,B} :=\tfrac{1}{2}(\Id+R_BR_A)
\end{empheq}
is the \emph{Douglas--Rachford operator} associated with the ordered pair $(A, B)$.
The sequence $(x_n)_\nnn$ in \eqref{e:DRAseq} is
called a \emph{DRA sequence with respect to $(A, B)$}, with starting point $x_0$.
By \cref{f:proj}\ref{f:proj_cvx}, when $A$ and $B$ are convex, then $P_A$, $P_B$ and hence $T_{A,B}$ are single-valued.
Notice that
\begin{equation}
(\forall x \in X)\quad T_{A,B}x =\tfrac{1}{2}(\Id+R_BR_A)x =\menge{x -a +P_B(2a -x)}{a \in P_Ax},
\end{equation}
and if $P_A$ is single-valued then
\begin{equation}
T_{A,B} =\tfrac{1}{2}(\Id+R_BR_A) =\Id -P_A +P_BR_A.
\end{equation}
In the sequel we adopt the convention that in the case where $P_Ax$ is not a singleton, 
$(P_Ax, P_BR_Ax) =\menge{(a, P_B(2a -x))}{a \in P_Ax}$.

The set of fixed points of $T_{A,B}$ is defined by $\Fix T_{A,B} :=\menge{x\in X}{x \in T_{A,B}x}$.
It follows from $T_{A,B}x =x -P_Ax +P_BR_Ax$ that
\begin{equation}
\label{e:xFixT}
x \in \Fix T_{A,B} \quad\Leftrightarrow\quad P_Ax\cap P_BR_Ax\neq\varnothing, 
\end{equation}
and that modified for clarity
\begin{equation}
\label{e:xFixT'}
\left.
\begin{array}{c}
x \in \Fix T_{A,B}\\
P_Ax \text{~is a singleton~}
\end{array}\right\} \quad\Rightarrow\quad P_Ax \in A\cap B. 
\end{equation}

For the convex case, the basic convergence result of the DRA sequence $(x_n)_\nnn$ 
and the ``shadow sequence'' $(P_Ax_n)_\nnn$ is as follows.
\begin{fact}[Convergence of DRA in the convex consistent case]
\label{f:cvg}
Let $A$ and $B$ be closed convex subsets of $X$ with $A\cap B\neq\varnothing$, 
and let $(x_n)_\nnn$ be a DRA sequence with respect to $(A, B)$.
Then the following hold:
\begin{enumerate}
\item
\label{f:cvg_FixT} 
$x_n \to x \in \Fix T_{A,B} =(A\cap B) +N_{A-B}(0)$ and $P_Ax_n \to P_Ax \in A\cap B$.
\item 
\label{f:cvg_AnB}
If $0 \in \inte(A -B)$, then $x_n \to x \in A\cap B$; 
the convergence is finite provided that $x \in A\cap \inte B$.
\end{enumerate}
\end{fact}
\begin{proof}
\ref{f:cvg_FixT}: This follows from \cite[Theorem~1]{LM79} and \cite[Theorem~1]{Sva11}; see also \cite[Corollary~3.9 and Theorem~3.13]{BCL04}.
\ref{f:cvg_AnB}: Clear from \cite[Lemma~3.2]{BDNP15}.
\end{proof}

\section{Locally identical sets}

\label{s:lis}

\begin{definition}
\label{d:local}
Let $A$ and $B$ be subsets of $X$ such that $A\cap B \neq\varnothing$.
Then $A$ and $B$ are called \emph{locally identical} around $c \in A\cap B$ if there exists $\varepsilon \in \RPP$ such that 
$A\cap \ball{c}{\varepsilon} =B\cap \ball{c}{\varepsilon}$.
We say that $A$ and $B$ are locally identical around a set $C \subseteq A\cap B$
if they are locally identical around every point in $C$.
When $A$ and $B$ are locally identical around a point $c$ (respectively, a set $C$), 
we also say that $(A, B)$ is locally identical around $c$ (respectively, $C$).
\end{definition}

\begin{lemma}
\label{l:id}
Let $A$ and $B$ be subsets of $X$ such that $A\cap B \neq\varnothing$. Then the following hold: 
\begin{enumerate}
\item
\label{l:id_int}
$A$ and $B$ are locally identical around $\inte(A\cap B)$. 
\item
\label{l:id_intersect}
If $A$ and $B$ are locally identical around $c \in A\cap B$, then $A$, $B$ and $A\cap B$ are also locally identical around $c$.
\item
\label{l:id_sup}
If $A \subseteq B$, and $c$ is a point in $A$ such that $d_{B\smallsetminus A}(c) >0$, then $A$ and $B$ are locally identical around $c$.    
\item
\label{l:id_sub}
If $A$ is closed convex, and $C$ is a closed subset of $A$ such that $A$ and $C$ are locally identical around $C$, then $A =C$.    
\item
\label{l:id_cvx}
If $A$ and $B$ are closed convex and locally identical around $A\cap B$, then $A =B$. 
\end{enumerate}
\end{lemma}
\begin{proof}
\ref{l:id_int}: 
Let $c \in \inte(A\cap B)$. Then there exists $\varepsilon \in \RPP$ such that $\ball{c}{\varepsilon} \subseteq A\cap B$, 
which implies $A\cap \ball{c}{\varepsilon} =\ball{c}{\varepsilon} =B\cap \ball{c}{\varepsilon}$, 
so $A$ and $B$ are locally identical around $c$.

\ref{l:id_intersect}: Note that if $A\cap \ball{c}{\varepsilon} =B\cap \ball{c}{\varepsilon}$ then  
$A\cap \ball{c}{\varepsilon} =B\cap \ball{c}{\varepsilon} =(A\cap B)\cap \ball{c}{\varepsilon}$. 

\ref{l:id_sup}: Since $d_{B\smallsetminus A}(c) >0$, there exists $\varepsilon \in \RPP$ such that
$(B\smallsetminus A)\cap \ball{c}{\varepsilon} =\varnothing$. 
Combining with $A \subseteq B$, we get 
$A\cap \ball{c}{\varepsilon} =(A\cap \ball{c}{\varepsilon})\cup ((B\smallsetminus A)\cap \ball{c}{\varepsilon}) =B \cap\ball{c}{\varepsilon}$.

\ref{l:id_sub}: Let $c \in C$. It suffices to show that
\begin{equation}
\label{e:allball}
(\forall\varepsilon \in \RPP)\quad A\cap \ball{c}{\varepsilon} =C\cap \ball{c}{\varepsilon}.
\end{equation}
Suppose to the contrary that \eqref{e:allball} does not hold. 
Since $A$ and $C$ are locally identical around $C$ which includes $c$, 
\begin{equation}
0 <\bar\varepsilon :=\sup\menge{\varepsilon \in \RPP}{A\cap \ball{c}{\varepsilon} =C\cap \ball{c}{\varepsilon}} <+\infty.
\end{equation}
Then $(\forall\varepsilon \in\ ]\bar\varepsilon, +\infty[)$ $A\cap \ball{c}{\varepsilon} \supsetneqq C\cap \ball{c}{\varepsilon}$.
Now let $\varepsilon_n \downarrow \bar\varepsilon$ and 
\begin{equation}
(\forall\nnn)\quad a_n \in A\cap \ball{c}{\varepsilon_n} \smallsetminus C.
\end{equation} 
By the boundedness of $(a_n)_\nnn$ and the closedness of $A$, we assume without loss of generality that $a_n \to a \in A$.
It follows from $\|a_n -c\| \leq \varepsilon_n$ that $\varepsilon :=\|a -c\| \leq \bar\varepsilon$. 
By the convexity of $A$, $(\forall \lambda \in ]0, 1[)$ $a_\lambda =\lambda a +(1 -\lambda)c \in A$, 
and $\|a_\lambda -c\| =\lambda \|a -c\| =\lambda\varepsilon <\bar\varepsilon$,
which yields $a_\lambda \in A\cap \ball{c}{\lambda\varepsilon} =C\cap \ball{c}{\lambda\varepsilon}$,
using the definition of $\bar\varepsilon$. 
From $a_\lambda \in C$ and the closedness of $C$, letting $\lambda \to 1^-$, 
we obtain $a \in C$, thus $A$ and $C$ are locally identical around $a$, 
i.e., $A\cap \ball{a}{\rho} =C\cap \ball{a}{\rho}$ for some $\rho \in \RPP$.
Since $a_n \to a$, we find $n_0 \in \NN$ satisfying $a_{n_0} \in \ball{a}{\rho}$.
Then $a_{n_0} \in A\cap \ball{a}{\rho} =C\cap \ball{a}{\rho} \subseteq C$,
which contradicts the fact that $(\forall\nnn)$ $a_n \notin C$. 
Therefore, \eqref{e:allball} holds.

Now pick an arbitrary $a \in A$, and let $\varepsilon >\|a -c\|$. 
By combining with \eqref{e:allball}, $a \in A\cap \ball{c}{\varepsilon} =C\cap \ball{c}{\varepsilon}$, and so $a \in C$. 
It follows that $A \subseteq C \subseteq A$, which gives $A =C$.

\ref{l:id_cvx}: Set $C :=A\cap B$. Then $C$ is closed, $C\subseteq A$, $C\subseteq B$, 
and by \ref{l:id_intersect}, $A$, $B$ and $C$ are locally identical around $C$. 
Now apply \ref{l:id_sub}.
\end{proof}

The following example illustrates that the assumption on convexity of $A$ in \cref{l:id}\ref{l:id_sub} is important.  
\begin{example}
Suppose that $X =\RR$, that $A =\{0, 1\}$ and that $C =\{0\}$. 
Then $A$ and $C$ are closed and locally identical around $C$, and $C \subseteq A$, but $C \neq A$.
This does not contradict \cref{l:id}\ref{l:id_sub} because $A$ is not convex.
\end{example}

\begin{lemma}
\label{l:proj}
Let $A$ and $B$ be closed subsets of $X$, and assume that $A$ and $B$ are locally identical around some $c \in A\cap B$,
say there exists $\varepsilon \in \RPP$ such that $A\cap \ball{c}{\varepsilon} =B\cap \ball{c}{\varepsilon}$. 
Let 
\begin{equation}
p \in A\cap \inte(\ball{c}{\varepsilon}) =B\cap \inte(\ball{c}{\varepsilon}).
\end{equation} 
Then the following hold:
\begin{enumerate}
\item 
\label{l:proj_subs} 
If $A \subseteq B$, then $(\forall x \in X)$ $P_Bx\cap \ball{c}{\varepsilon} \subseteq P_Ax$.
\item 
\label{l:proj_cvx} 
If $A$ and $B$ are convex, then $(\forall x \in X)$ $p =P_Ax$ $\Leftrightarrow$ $p =P_Bx$.
Equivalently, if $A$ and $B$ are convex then $P_A^{-1}p = P_B^{-1}p$.   
\item 
\label{l:proj_sub} 
If $A \subseteq B$ and $B$ is convex, then $(\forall x \in X)$ 
\begin{enumerate}
\item\label{l:proj_sub.a} $P_Bx \in \ball{c}{\varepsilon} \Rightarrow P_Ax =P_Bx$;
\item\label{l:proj_sub.b} $p \in P_Ax \quad\Leftrightarrow\quad p =P_Bx$; 
\item\label{l:proj_sub.c} $p \in P_Ax \Rightarrow P_Ax =P_Bx =p$.
\end{enumerate}
\end{enumerate}
\end{lemma}
\begin{proof}
\ref{l:proj_subs}: Observe that $P_Bx\cap \ball{c}{\varepsilon} =P_Bx\cap (B\cap \ball{c}{\varepsilon}) =P_Bx\cap (A\cap \ball{c}{\varepsilon}) \subseteq A\cap P_Bx$. The conclusion follows \cref{l:projAB}\ref{l:projAB_sub}.

To prove \ref{l:proj_cvx} and \ref{l:proj_sub}, note that since $p \in \inte(\ball{c}{\varepsilon})$, 
there exists $\rho \in \RPP$ such that $\ball{p}{\rho} \subseteq \ball{c}{\varepsilon}$,
which yields
\begin{equation}
\label{e:smallball}
A\cap \ball{p}{\rho} =B\cap \ball{p}{\rho}.
\end{equation}  

\ref{l:proj_cvx}: By \cite[Lemma~2.12]{BDNP15}, 
it follows from $p \in A\cap B$ and \eqref{e:smallball} that $N_A(p) =N_B(p)$. 
Now using \eqref{e:proj_cvx}, 
$(\forall x \in X)$ $p =P_Ax$ $\Leftrightarrow$ $x -p \in N_A(p) =N_B(p)$ $\Leftrightarrow$ $p =P_Bx$. 
Hence, $P_A^{-1}p = P_B^{-1}p$.

\ref{l:proj_sub.a}: Let $x \in X$. Assume that $P_Bx \in \ball{c}{\varepsilon}$.
Then $P_Bx \in B\cap \ball{c}{\varepsilon} =A\cap \ball{c}{\varepsilon} \subseteq A$. 
By \cref{l:projAB}\ref{l:projAB_cvx}, $P_Ax =P_Bx$. 

\ref{l:proj_sub.b}: Using \eqref{e:smallball} and applying \ref{l:proj_cvx} 
for two convex sets $A\cap \ball{p}{\rho}$ and $B$, we obtain 
$P_{A\cap \ball{p}{\rho}}^{-1}p =P_B^{-1}p$.
Next applying \cref{l:projAB}\ref{l:projAB_preimage} for $A\cap \ball{p}{\rho} \subseteq A$ and $A \subseteq B$, we have
$P_A^{-1}p \subseteq P_{A\cap \ball{p}{\rho}}^{-1}p =P_B^{-1}p \subseteq P_A^{-1}p$, and so $P_A^{-1}p =P_B^{-1}p$.

\ref{l:proj_sub.c}: Now assume $p \in P_Ax$. 
Then \ref{l:proj_sub.b} gives $P_Bx =p \in A$, and \cref{l:projAB}\ref{l:projAB_cvx} gives $P_Ax =P_Bx =p$.
\end{proof}

\section{Cases involving halfspaces}

\label{s:4}

In this section, we assume that 
\begin{empheq}[box=\mybluebox]{equation}
\text{$f \colon X \to \RR$ is convex and continuous},
\end{empheq}
and that
\begin{empheq}[box=\mybluebox]{equation}
\epi f :=\menge{(x, \rho) \in X\times \RR}{f(x)\leq\rho}.
\end{empheq}
In the space $X\times \RR$, we set
\begin{empheq}[box=\mybluebox]{equation}
H :=X\times \{0\} \quad\text{and}\quad B :=\epi f.
\end{empheq}
Then the projection onto $H$ is given by
\begin{equation}
(\forall (x, \rho) \in X\times\RR)\quad P_H(x, \rho) =(x, 0),
\end{equation}
the projection onto $B$ is described as in \cref{ex:proj_epi},
and the effect of performing each step of the DRA applied to $H$ and $B$ is characterized in the following result.
\begin{fact}[One DRA step]
\label{f:DRstep}
Let $z =(x, \rho) \in X\times \RR$, 
and set $z_+ :=(x_+, \rho_+) =T_{H,B}(x, \rho)$. 
Then the following hold:
\begin{enumerate}
\item 
\label{f:DRstep_form}
If $\rho \leq -f(x)$, then $z_+ =(x, 0) \in H$. Otherwise, there exists $x_+^* \in \partial f(x_+)$ such that 
\begin{equation}
\label{e:DRstep}
x_+ =x -\rho_+x_+^*,\; f(x_+) \leq f(x), \;\text{and}\;
\rho_+ =\rho +f(x_+) >0;
\end{equation}
in which either ($\rho\geq 0$ and $z_+ \in B$)  
or ($\rho<0$ and $T_{H,B}z_+ \in B$). 
\item 
\label{f:DRstep_pos}
$\ran T_{H,B} \subseteq X\times \RP$, or equivalently, $(\forall z \in X\times \RR)$ $z_+ \in X\times \RP$. 
\end{enumerate}
\end{fact}
\begin{proof}
\ref{f:DRstep_form}: \cite[Corollary~5.3(i)\&(ii)]{BDNP15}.
\ref{f:DRstep_pos}: Clear from \ref{f:DRstep_form}.  
\end{proof}

We have the following result on convergence of the DRA in the case of a hyperplane and an epigraph.
\begin{fact}[Finite convergence of DRA in the
(hyperplane,epigraph) case]
\label{f:P-E}
Suppose that
\begin{equation}
A = H \quad\text{and}\quad B =\epi f \text{~with~} \inf_X f <0.
\end{equation}
Given a starting point $z_0 =(x_0, \rho_0) \in X\times \RR$, generate the DRA sequence $(z_n)_\nnn$ by 
\begin{equation}
\label{e:DRAseq'}
(\forall\nnn)\quad z_{n+1} =(x_{n+1}, \rho_{n+1}) =T_{A,B}z_n.
\end{equation}
Then $(z_n)_\nnn$ converges \emph{finitely} to a point in $A\cap B$.
\end{fact}
\begin{proof}
See \cite[Theorem~5.4]{BDNP15}.
\end{proof}

In view of \cref{f:P-E}, it is natural to ask about the convergence of the DRA when $A$ is a halfspace instead of a hyperplane.
\begin{theorem}[Finite convergence of DRA in the
(halfspace,epigraph) case]
\label{t:H-E}
Suppose that either 
\begin{enumerate}
\item
\label{t:H-E+}
$A = H_+ :=X\times \RP$ and $B =\epi f$,
or 
\item
\label{t:H-E-}
$A = H_- :=X\times \RM$ and $B =\epi f$ with $\inf_X f <0$.
\end{enumerate}
Then the DRA sequence \eqref{e:DRAseq'} converges \emph{finitely} to a point in $A\cap B$.
\end{theorem}
\begin{proof}
\ref{t:H-E+}: Let $z =(x, \rho) \in X\times \RR$. 
If $z \in H_-$, then $P_Az =P_Hz$, and so $z_+ :=T_{A,B}z =T_{H,B}z \in H_+$ due to \cref{f:DRstep}\ref{f:DRstep_pos}.
If $z \in H_+ \cap B =A\cap B$, we are done.
If $z \in H_+ \smallsetminus B$, then $P_Az =z$, $R_Az =z$, and by \cref{ex:proj_epi}, $P_BR_Az =P_Bz =(x_+, f(x_+))$ with $f(x_+) >\rho \geq 0$, which implies $z_+ =z -P_Az +P_BR_Az =(x_+, f(x_+)) \in H_+ \cap B =A\cap B$. We deduce that the DRA sequence \eqref{e:DRAseq'} converges in at most two steps. 

\ref{t:H-E-}: If $z_0 \in H_- =A$, then $P_Az_0 =z_0$, $R_Az_0 =z_0$, and $ z_1 =P_BR_Az_0 =(x_1, f(x_1)) \in B$, 
which gives $z_1 \in A\cap B$ if $f(x_1) \leq 0$, and $z_1 \in H_+$ otherwise. 
It is thus sufficient to consider the case $z_0 \in H_+$. 
Then $P_Az_0 =P_Hz_0$, and so $z_1 =T_{A,B}z_0 =T_{H,B}z_0 \in H_+$ due to \cref{f:DRstep}\ref{f:DRstep_pos}.
This implies that 
\begin{equation}
(\forall\nnn)\quad z_n \in H_+ \quad\text{and}\quad z_{n+1} =T_{H,B}z_n.
\end{equation}
Now apply \cref{f:P-E}.
\end{proof}

The following example whose special cases can be found in \cite{BDNP16} illustrates that the Slater's condition $\inf_X f <0$ in \cref{f:P-E} and \cref{t:H-E}\ref{t:H-E-} is important. 
\begin{example}
\label{ex:epi_infinite}
Suppose that either $A =H$ or $A =H_- :=X\times \RM$, that $B =\epi f$ with $\inf_X f \geq 0$,
and that $f$ is differentiable at its minimizers (if they exist). 
Let $z_0 =(x_0, \rho_0) \in B$, where $x_0$ is not a minimizer of $f$,
and generate the DRA sequence $(z_n)_\nnn$ as in \eqref{e:DRAseq'}.
Then $(P_Az_n)_\nnn$ and thus also $(z_n)_\nnn$ do not converge finitely.
\end{example}
\begin{proof}
Firstly, we claim that if $z =(x, \rho) \in B$, where $x$ is not a minimizer of $f$,   
then $z_+ :=T_{A,B}z =T_{H,B}z =(x_+, \rho_+) \in B$ and $x_+$ is not a minimizer of $f$. 
Indeed, by assumption, $\rho >0$, so $P_Az =P_Hz$, and then $z_+ =T_{A,B}z =T_{H,B}z$. 
By using \cref{f:DRstep}\ref{f:DRstep_form},
$z_+ \in B$ and 
\begin{equation}
\label{e:decrease}
x_+ =x -\rho_+x_+^* \quad\text{with}\quad x_+^* \in \partial f(x_+),               
\quad\text{and}\quad \rho_+ =\rho +f(x_+) >0. 
\end{equation}
If $x_+$ is a minimizer of $f$, then $x_+^* =\nabla f(x_+) =0$, 
and by \cref{e:decrease}, $x =x_+$ is a minimizer, which is absurd.
Hence, the claim holds. 
As a result, 
\begin{equation}
\label{e:xn}
(\forall\nnn)\quad \text{$x_n$ is not a minimizer of $f$}.
\end{equation}
Now assume that $(P_Az_n)_\nnn =(x_n, 0)_\nnn$ converges finitely. 
Then there exists $n \in \NN$ such that $x_{n+1} =x_n$.
Using again \cref{e:decrease}, we get $x_{n+1}^* =0 \in \partial f(x_{n+1})$, 
which contradicts \cref{e:xn}.   
\end{proof}

\begin{theorem}[Finite convergence of DRA in (hyperplane or
halfspace,halfspace) case]
\label{t:H-H}
Suppose that $A$ is either a hyperplane or a halfspace, that $B$ is a halfspace of $X$, 
and that $A\cap B \neq\varnothing$.
Then every DRA sequence $(x_n)_\nnn$ with respect to $(A, B)$ converges \emph{finitely} to a point $x$,
where $x \in A\cap B$ or $(\forall\nnn)$ $x_n =x \in B$ with $P_Ax \in A\cap B$. 
\end{theorem}
\begin{proof}
If $\dim X =0$, i.e., $X =\{0\}$, then the result is trivial,
so we will work in the space $X\times \RR$ with $\dim X \geq 0$, and denote by $(z_n)_\nnn$ the DRA sequence.
After rotating the sets if necessary, we can and do assume that $A =X\times \RM$,
and $B =\menge{(x, \rho) \in X\times \RR}{\scal{(x, \rho)}{(u, \nu)} \leq \eta}$, 
with $(u, \nu) \in X\times \RR\smallsetminus \{(0, 0)\}$ and $\eta \in \RR$.
Noting that $\scal{(x, \rho)}{(u, \nu)} =\scal{x}{u} +\rho\nu$, 
we distinguish the following three cases.

\emph{Case 1:} $\nu <0$. Then 
\begin{equation}
B =\menge{(x, \rho) \in X\times \RR}{\frac{\eta -\scal{x}{u}}{\nu} \leq \rho}
\end{equation}
is the epigraph of the linear function 
\begin{equation}
f\colon X \to \RR\colon x \mapsto \frac{\eta -\scal{x}{u}}{\nu}.
\end{equation}
If $\inf_X f <0$ , we are done due to \cref{t:H-E}\ref{t:H-E-}. 
Assume that $\inf_X f \geq 0$. Then $u = 0 \in X$ since $u \in X\smallsetminus \{0\}$ implies
$\inf_X f \leq \inf_{\lambda \in \RM} f(\lambda u) =\inf_{\lambda \in \RM} \frac{\eta -\lambda\|u\|^2}{v} =-\infty$.
Now in turn, $(\forall x \in X)$ $f(x) =\frac{\eta}{\nu}$, and so $\frac{\eta}{\nu} =\inf_X f \geq 0$, which gives $\eta \leq 0$. 
By the assumption that $A\cap B \neq\varnothing$, we must have $\eta =0$, and then $B =X\times \RP$.
Let $z =(x, \rho) \in X\times \RR$. 
If $z \in B$, then $R_Az =(x, -\rho)$, and $R_BR_Az =(x, \rho) =z$, which gives $T_{A,B}z =z$,
i.e., $z \in \Fix T_{A,B}$, in which case $P_Az =(x, 0) \in A\cap B$. 
If $z \notin B$, then $z \in A$ and $R_Az =z$, $R_BR_Az =R_Bz =(x, -\rho)$, so 
\begin{equation}
T_{A,B}z =\tfrac{1}{2}(z +R_BR_Az) =(x, 0) \in A\cap B.
\end{equation} 

\emph{Case 2:} $\nu >0$. Then 
\begin{equation}
B =\menge{(x, \rho) \in X\times \RR}{\frac{\eta -\scal{x}{u}}{\nu} \geq \rho}.
\end{equation}
After reflecting the sets across the hyperplane $X \times \{0\}$,
we have $A =X\times \RP$, and $B$ is the epigraph of a linear function. 
Now apply \cref{t:H-E}\ref{t:H-E+}.

\emph{Case 3:} $\nu =0$. Then $u \in X\smallsetminus \{0\}$ and
\begin{equation}
B =\menge{(x, \rho) \in X\times \RR}{\scal{x}{u} \leq \eta} =\menge{x \in X}{\scal{x}{u} \leq \eta}\times \RR.
\end{equation}
Let $z =(x, \rho) \in X\times \RR$. If $z \in A\cap B$, we are done.
If $z \in A\smallsetminus B$, then $\rho \in \RM$, $R_Az =P_Az =z \notin B$, 
and by \cref{ex:proj_h}\ref{ex:proj_half}, 
\begin{equation}
P_BR_Az =P_Bz =\left(x -\frac{\scal{x}{u} -\eta}{\|u\|^2}u, \rho\right) \in B,
\end{equation}
which is also in $A =X\times \RM$ and which yields
\begin{equation}
T_{A,B}z =z -P_Az +P_BR_Az =P_BR_Az \in A\cap B. 
\end{equation}  
Now assume that $z \notin A$. We have $P_Az =(x, 0)$ and $R_Az =(x, -\rho)$. 
If $(x, -\rho) \in B$, then $R_BR_Az =(x, -\rho)$, 
and $T_{A,B}z =\frac{1}{2}(z +R_BR_Az) =(x, 0) \in A\cap B$.  
Finally, if $(x, -\rho) \notin B$, then again by \cref{ex:proj_h}\ref{ex:proj_half}, 
\begin{equation}
P_BR_Az =P_B(x, -\rho) =(x, -\rho) -\frac{\scal{x}{u} -\eta}{\|u\|^2}(u, 0),
\end{equation}  
and thus,
\begin{equation}
T_{A,B}z =z -P_Az +P_BR_Az =\left(x -\frac{\scal{x}{u} -\eta}{\|u\|^2}u, 0\right) \in A. 
\end{equation}  
Moreover, $\scal{x -\frac{\scal{x}{u} -\eta}{\|u\|^2}u}{u} =\eta$, so $T_{A,B}z \in B$, 
and we get $T_{A,B}z \in A\cap B$. 

The proof for the (hyperplane,halfspace) case is similar and uses \cref{f:P-E}.
\end{proof}

\section{Expanding and modifying sets}

\label{s:expand}

\begin{lemma}[Expanding sets]
\label{l:expand}
Let $A$ and $B$ be closed (not necessarily convex) subsets of $X$ such that $A\cap B \neq\varnothing$, 
and let $x_0$ be in $X$.
Suppose that the DRA sequence $(x_n)_\nnn$ with respect to $(A, B)$, with starting point $x_0$, converges to $x \in \Fix T_{A,B}$.
Suppose further that there exist two closed convex sets $A'$ and $B'$ in $X$ such that
$A\subseteq A'$, $B\subseteq B'$, 
and that both $(A, A')$ and $(B, B')$ are locally identical around some $c \in P_Ax$.  
Then $P_Ax =P_{A'}x$, $x \in \Fix T_{A',B'}$ and 
\begin{equation}
\label{e:adjust}
(\exists n_0 \in \NN)(\forall n \geq n_0)\quad x_{n+1} =T_{A',B'}x_n,
\end{equation}
i.e., $(\exists n_0 \in \NN)(\forall\nnn)\quad T_{A,B}^nx_{n_0} =T_{A',B'}^nx_{n_0}$.
\end{lemma}
\begin{proof}
By assumption and \cref{l:proj}\ref{l:proj_sub.c}, $P_Ax =P_{A'}x =c$, and so $R_Ax =R_{A'}x$.
Since $x \in \Fix T_{A,B}$, it follows from \eqref{e:xFixT} that $c \in P_BR_Ax =P_BR_{A'}x$. 
Using again \cref{l:proj}\ref{l:proj_sub.c}, $P_BR_{A'}x =P_{B'}R_{A'}x =c$.
We get $P_{A'}x =P_{B'}R_{A'}x =c$, and again by \eqref{e:xFixT}, $x \in \Fix T_{A',B'}$. 
Now by the definition of $A'$ and $B'$, there exists $\varepsilon \in \RPP$ such that 
\begin{equation}
A\cap \ball{c}{\varepsilon} =A'\cap \ball{c}{\varepsilon}
\quad\text{and}\quad
B\cap \ball{c}{\varepsilon} =B'\cap \ball{c}{\varepsilon}.
\end{equation}
There exists $n_0 \in \NN$ such that
\begin{equation}
(\forall n\geq n_0)\quad \|x_n -x\| <\varepsilon.
\end{equation}
Let $n \geq n_0$. Since $P_{A'}$, $P_{B'}$ are (firmly) nonexpansive 
and $R_{A'}$ is nonexpansive (\cref{f:proj}\ref{f:proj_firm}\&\ref{f:proj_nonex}), 
\begin{equation}
\|P_{A'}x_n -c\| =\|P_{A'}x_n -P_{A'}x\| \leq \|x_n -x\| <\varepsilon,
\end{equation}
and also
\begin{equation}
\|P_{B'}R_{A'}x_n -c\| =\|P_{B'}R_{A'}x_n -P_{B'}R_{A'}x\| \leq \|x_n -x\| <\varepsilon.
\end{equation}
Thus, $P_{A'}x_n \in \ball{c}{\varepsilon}$ and $P_{B'}R_{A'}x_n \in \ball{c}{\varepsilon}$.
By \cref{l:proj}\ref{l:proj_sub.a}, $P_Ax_n =P_{A'}x_n$ and $P_BR_{A'}x_n =P_{B'}R_{A'}x_n$, 
which implies $R_Ax_n =R_{A'}x_n$ and $P_BR_Ax_n =P_{B'}R_{A'}x_n$.  
We deduce that $x_{n+1} =T_{A,B}x_n =T_{A',B'}x_n$.
\end{proof}

If the assumption that $A\subseteq A'$ and $B\subseteq B'$ in \cref{l:expand} 
is replaced by the assumption on convexity of $A$ and $B$, then \eqref{e:adjust} still holds, as shown in the following lemma. 
We shall now look at situations where $(A',B')$ are modifications
of $(A,B)$ that preserve local structure. 

\begin{lemma}
\label{l:adjust}
Let $A$ and $B$ be closed convex subsets of $X$ such that $A\cap B \neq\varnothing$, 
and let $(x_n)_\nnn$ be the DRA sequence with respect to $(A, B)$, with starting point $x_0 \in X$.
Suppose that there exist two closed convex sets $A'$ and $B'$ in $X$ such that
both $(A, A')$ and $(B, B')$ are locally identical around $P_Ax \in A\cap B$,
where $x \in \Fix T_{A,B}$ is the limit of $(x_n)_\nnn$. 
Then 
\begin{equation}
(\exists n_0 \in \NN)(\forall n \geq n_0)\quad x_{n+1} =T_{A',B'}x_n,
\end{equation}
i.e., $(\exists n_0 \in \NN)(\forall\nnn)\quad T_{A,B}^nx_{n_0} =T_{A',B'}^nx_{n_0}$.
\end{lemma}
\begin{proof}
Recall from \cref{f:cvg}\ref{f:cvg_FixT} that $x_n \to x \in \Fix T_{A,B}$ with $P_Ax \in A\cap B$.
Setting $c :=P_Ax =P_BR_Ax$, 
from the assumption on $A'$ and $B'$, there is $\varepsilon \in \RPP$ such that 
\begin{equation}
A\cap \ball{c}{\varepsilon} =A'\cap \ball{c}{\varepsilon}
\quad\text{and}\quad
B\cap \ball{c}{\varepsilon} =B'\cap \ball{c}{\varepsilon}.
\end{equation}
Furthermore, there exists $n_0 \in \NN$ such that
\begin{equation}
(\forall n\geq n_0)\quad \|x_n -x\| <\varepsilon.
\end{equation}
Let $n \geq n_0$. 
According to \cref{f:proj}\ref{f:proj_firm}\&\ref{f:proj_nonex}, $P_A$, $P_B$ are (firmly) nonexpansive 
and $R_A$ is nonexpansive, so  
\begin{equation}
\|P_Ax_n -c\| =\|P_Ax_n -P_Ax\| \leq \|x_n -x\| <\varepsilon,
\end{equation}
and also
\begin{equation}
\|P_BR_Ax_n -c\| =\|P_BR_Ax_n -P_BR_Ax\| \leq \|x_n -x\| <\varepsilon.
\end{equation}
Therefore, $P_Ax_n \in A\cap \inte\ball{c}{\varepsilon}$ and $P_BR_Ax_n \in B\cap\inte\ball{c}{\varepsilon}$.
Using \cref{l:proj}\ref{l:proj_cvx}, $P_Ax_n =P_{A'}x_n$ and $P_BR_Ax_n =P_{B'}R_Ax_n$. 
Hence $R_Ax_n =R_{A'}x_n$ and $P_BR_Ax_n =P_{B'}R_{A'}x_n$.  
We obtain that $x_{n+1} =T_{A,B}x_n =T_{A',B'}x_n$.
\end{proof}

\begin{theorem}[Modifying sets]
\label{t:global}
Let $A$ and $B$ be closed convex subsets of $X$ such that $A\cap B \neq\varnothing$.
Suppose that there exist two closed convex sets $A'$ and $B'$ in $X$ such that
both $(A, A')$ and $(B, B')$ are locally identical around $A\cap B$. 
Then 
for any DRA sequence $(x_n)_\nnn$ with respect to $(A, B)$ in \eqref{e:DRAseq},
\begin{equation}
\label{e:global}
(\exists n_0 \in \NN)(\forall n \geq n_0)\quad x_{n+1} =T_{A',B'}x_n,
\end{equation}
and this is still true when exchanging the roles of $T_{A,B}$ and $T_{A',B'}$ in \eqref{e:DRAseq} and \eqref{e:global}. 
\end{theorem}
\begin{proof}
By \cref{f:cvg}\ref{f:cvg_FixT}, $x_n \to x \in \Fix T_{A,B}$ with $P_Ax \in A\cap B$.
Now apply \cref{l:adjust}.

Let us exchange the roles of $T_{A,B}$ and $T_{A',B'}$ in \eqref{e:DRAseq} and \eqref{e:global}, 
i.e., $(\forall\nnn)$ $x_{n+1} =T_{A',B'}x_n$, and we shall prove that
\begin{equation}
\label{e:restrict'}
(\exists n_0 \in \NN)(\forall n \geq n_0)\quad x_{n+1} =T_{A,B}x_n.
\end{equation}
By the assumption on $A'$ and $B'$, we have $A\cap B \subseteq A'$, $A\cap B \subseteq B'$, 
and for all $c \in A\cap B$, there exists $\varepsilon \in \RPP$ such that
\begin{equation}
A\cap \ball{c}{\varepsilon} =A'\cap \ball{c}{\varepsilon}
\quad\text{and}\quad
B\cap \ball{c}{\varepsilon} =B'\cap \ball{c}{\varepsilon}.
\end{equation}
Then
\begin{equation}
(A\cap B)\cap \ball{c}{\varepsilon} =(A'\cap B')\cap \ball{c}{\varepsilon}.
\end{equation}
Therefore, $A\cap B$ and $A'\cap B'$ are locally identical around $A\cap B$.
Noting that $A\cap B$ and $A'\cap B'$ are closed convex, and $A\cap B \subseteq A'\cap B'$,
\cref{l:id}\ref{l:id_sub} gives $A\cap B =A'\cap B'$.
Next again by \cref{f:cvg}\ref{f:cvg_FixT}, $x_n \to x \in \Fix T_{A',B'}$ with $P_{A'}x \in A'\cap B' =A\cap B$.
By assumption, both $(A, A')$ and $(B, B')$ are locally identical around $P_{A'}x$,
and hence the proof is completed by applying \cref{l:adjust}. 
\end{proof}

In the following, we say that the DRA applied to $(A, B)$ \emph{converges finitely globally} 
if the sequence $(T_{A,B}^nx)_\nnn$ converges finitely for all $x \in X$.
\begin{theorem}
\label{t:finite}
Let $A$ and $B$ be nonempty closed convex subsets of $X$.
Then the DRA applied to $(A, B)$ converges \emph{finitely globally} provided one of the following holds:
\begin{enumerate}
\item 
\label{t:finite_AbdryB}
$A\cap B \neq\varnothing$ and $A\cap \bd B =\varnothing$;
equivalently, $A \subseteq \inte B$.
\item 
\label{t:finite_AB}
$A\cap \bd B \neq\varnothing$ and there exist two closed convex sets $A'$ and $B'$ in $X$ such that
both $(A, A')$ and $(B, B')$ are locally identical around $A\cap \bd B$, 
and that the DRA applied to $(A', B')$ converges finitely globally when $A'\cap B' \neq\varnothing$. 
\item 
\label{t:finite_AintB}
$A\cap \inte B \neq\varnothing$, $A\cap \bd B \neq\varnothing$ and there exist two closed convex sets $A'$ and $B'$ in $X$ such that
both $(A, A')$ and $(B, B')$ are locally identical around $A\cap \bd B$, 
and that the DRA applied to $(A', B')$ converges finitely globally when $A'\cap \inte B' \neq\varnothing$.
\end{enumerate}
\end{theorem}
\begin{proof}
Let $(x_n)_\nnn$ be a DRA sequence with respect to $(A, B)$. 

\ref{t:finite_AbdryB}: It follows from $A\cap B \neq\varnothing$, $A\cap \bd B =\varnothing$ and the closedness of $B$ that
$A\cap \inte B =A\cap B \neq\varnothing$, and so $0 \in \inte(A -B)$. 
By \cref{f:cvg}\ref{f:cvg_AnB}, $x_n \to x \in A\cap B =A\cap \inte B$ finitely.

Now if $A \subseteq \inte B$, then $A\cap B =A \neq \varnothing$, 
and $A\cap \bd B \subseteq \inte B\cap \bd B =\varnothing$, which implies $A\cap \bd B =\varnothing$.
Conversely, assume that $A\cap B \neq\varnothing$ and $A\cap \bd B =\varnothing$.
Let $a \in A$. Then $a \notin \bd B$. 
We have to show $a \in \inte B$. Suppose to the contrary that $a \notin \inte B$.
Pick $b \in A\cap B$. By convexity, $\left[a, b\right] :=\menge{\lambda a +(1 -\lambda)b}{0 \leq \lambda \leq 1}\subseteq A$,
and so $\left[a, b\right]\cap \bd B =\varnothing$, 
which is impossible since $a \notin B$ and $b \in B$.
Hence, $a \in \inte B$ for all $a \in A$. This means $A \subseteq \inte B$.

\ref{t:finite_AB}: By assumption, $A\cap \bd B \subseteq A'\cap B'$,
and so the DRA applied to $(A', B')$ converges finitely globally. 
If $A\cap \inte B =\varnothing$, 
then both $(A, A')$ and $(B, B')$ are locally identical around $A\cap \bd B =A\cap B$ (using the closedness of $B$),
and using \cref{t:global}, 
\begin{equation}
\label{e:exchange}
(\exists n_0 \in \NN)(\forall n \geq n_0)\quad x_{n+1} =T_{A',B'}x_n,
\end{equation}
which implies the finite convergence of $(x_n)_\nnn$ due to the finite convergence of the DRA applied to $(A', B')$.

Next assume that $A\cap \inte B \neq\varnothing$. 
Then $\inte(A -B) \neq\varnothing$. By \cref{f:cvg}\ref{f:cvg_AnB}, 
$x_n \to x \in A\cap B$, and this convergence is finite when $x \in A\cap \inte B$.
It thus suffices to consider the case when $x \in A\cap \bd B$.
Then $(A, A')$ and $(B, B')$ are locally identical around $x =P_Ax$, and by \cref{l:adjust},
\eqref{e:exchange} holds. Using again the finite convergence of the DRA applied to $(A', B')$, we are done.

\ref{t:finite_AintB}: First, we show that $A'\cap \inte B' \neq\varnothing$.
Let $c \in A\cap \bd B$. By the assumption on $A'$ and $B'$, there is $\varepsilon \in \RPP$ such that 
\begin{equation}
A\cap \ball{c}{\varepsilon} =A'\cap \ball{c}{\varepsilon}
\quad\text{and}\quad
B\cap \ball{c}{\varepsilon} =B'\cap \ball{c}{\varepsilon}.
\end{equation}
Now let $d \in A\cap \inte B$. Then $c \neq d$, and by the convexity of $A$ and $B$, 
\cite[Proposition~3.35]{BC11} implies $\left]c, d\right] :=\menge{\lambda c +(1 -\lambda)d}{0 \leq \lambda <1} \subseteq A\cap \inte B$.
Therefore, $\left]c, d\right] \cap \inte\ball{c}{\varepsilon} \subseteq A\cap \ball{c}{\varepsilon} =A'\cap \ball{c}{\varepsilon} \subseteq A'$ and $\left]c, d\right] \cap \inte\ball{c}{\varepsilon} \subseteq \inte(B\cap \ball{c}{\varepsilon}) =\inte(B'\cap \ball{c}{\varepsilon}) \subseteq \inte B'$.
We deduce that $A'\cap \inte B' \neq\varnothing$. 
By assumption, the DRA applied to $(A', B')$ converges finitely globally.
Now argue as the case where $A\cap \inte B \neq\varnothing$ in the proof of part \ref{t:finite_AB}.  
\end{proof}

\begin{corollary}
\label{c:linear}
Let $A$ and $B$ be closed convex subsets of $X$ such that $A\cap B \neq\varnothing$.
Suppose that both $(A, \aff A)$ and $(B, \aff B)$ are locally identical around $A\cap \bd B$ when $A\cap \bd B \neq\varnothing$. 
Then every DRA sequence $(x_n)_\nnn$ with respect to $(A, B)$ converges \emph{linearly} 
with rate $c_F(\aff A -\aff A, \aff B -\aff B)$ to a point $x \in \Fix T_{A,B}$ with $P_Ax \in A\cap B$,
where $c_F(U, V)$ is the \emph{cosine of the Friedrichs angle} between two subspaces $U$ and $V$ defined by
\begin{equation}
c_F(U, V) :=\sup\menge{|\scal{u}{v}|}{u \in U\cap(U\cap V)^\perp, v \in V\cap(U\cap V)^\perp, \|u\| \leq 1, \|v\| \leq 1}.
\end{equation}
\end{corollary}
\begin{proof}
If $A\cap \bd B =\varnothing$, then by \cref{t:finite}\ref{t:finite_AbdryB}, we are done.
Now assume that $A\cap \bd B \neq\varnothing$. 
By assumption and \cref{t:global},
\begin{equation}
(\exists n_0 \in \NN)(\forall n \geq n_0)\quad x_{n+1} =T_{\aff A,\aff B}x_n.
\end{equation}
Since we work with a finite-dimensional space, \cite[Corollary~4.5]{BBNPW14} completes the proof.
\end{proof}

\begin{example}
Suppose that $X =\RR^3$, that $A =\left[(2, 1, 2), (-2, 1, -2)\right]$, 
and that $B =\menge{(\alpha, \beta, \gamma) \in \RR^3}{|\alpha| \leq 2, |\beta| \leq 2, \gamma =1}$.
Then $A\cap B =\{(1, 1, 1)\} \in \reli A\cap \reli B$.
By \cite[Theorem~4.14]{Pha16}, every DRA sequence with respect to $(A, B)$ converges linearly.
Furthermore, 
$\aff A =\menge{(\alpha, \beta, \gamma) \in \RR^3}{\alpha - \gamma =0, \beta =1}$, 
$\aff B =\menge{(\alpha, \beta, \gamma) \in \RR^3}{\gamma =1}$, 
$\aff A -\aff A =\menge{(\alpha, \beta, \gamma) \in \RR^3}{\alpha - \gamma =0, \beta =0}$,
$\aff B -\aff B =\menge{(\alpha, \beta, \gamma) \in \RR^3}{\gamma =0}$,
and both $(A, \aff A)$ and $(B, \aff B)$ are locally identical around $A\cap \bd B =A\cap B$.
By applying \cref{c:linear}, the linearly rate is $c_F(\aff A
-\aff A, \aff B -\aff B) =1/\sqrt{2}$. 
\end{example}

\begin{proposition}[Finite convergence of the DRA in the
(hyperplane or halfspace,ball) case]
\label{p:H-B}
Let $A$ be either a hyperplane or a halfspace, and $B$ be a closed ball of $X$ such that $A\cap \inte B \neq\varnothing$.
Then every DRA sequence $(x_n)_\nnn$ with respect to $(A, B)$ converges in \emph{finitely many steps} to a point in $A\cap B$.
\end{proposition}
\begin{proof}
If $\dim X =0$, i.e., $X =\{0\}$, then the result is trivial,
so we will work in the space $X\times \RR$ with $\dim X \geq 0$, and denote by $(z_n)_\nnn$ the DRA sequence.
We just prove the the result for the case when $A$ is a
hyperplane because the case when $A$ is a halfspace is similar.
Without loss of generality, we assume that $A =X\times \{0\}$ and that $B =\ball{(0,\theta)}{1}$ is the closed ball of radius $1$ and center $(0, \theta) \in X\times \RR$ with 
$0 \leq \theta < 1$.
Nothing that 
\begin{equation}
B =\menge{(x, \rho) \in X\times \RR}{\theta -\sqrt{1 -\|x\|^2} \leq \rho \leq \theta +\sqrt{1 -\|x\|^2}},
\end{equation}
we write $B =B_-\cup B_+$, where
\begin{subequations}
\begin{align}
B_- &=\menge{(x, \rho) \in X\times \RR}{\theta -\sqrt{1 -\|x\|^2} \leq \rho \leq \theta}, \\
B_+ &=\menge{(x, \rho) \in X\times \RR}{\theta \leq \rho \leq \theta +\sqrt{1 -\|x\|^2}}.
\end{align}
\end{subequations}
We distinguish two cases.

\emph{Case 1:} $\theta =0$. Then the two halves $B_- \subseteq X\times \RM$ and $B_+ \subseteq X\times \RP$ of the ball $B$ are symmetric with respect to the hyperplane $A$. 
By symmetry, we can and do assume that $z_0 =(x_0, \rho_0) \in X\times \RP$.
Now for any $z =(x, \rho) \in X\times \RP$, we have $P_Az =(x, 0)$, $R_Az =(x, -\rho)$, and by \cref{ex:proj_ball},  
\begin{equation}
P_BR_Az =\delta(x, -\rho)\quad\text{with}\quad \delta :=\frac{1}{\max\{\sqrt{\|x\|^2 +\rho^2}, 1\}} \leq 1,
\end{equation}
which gives 
\begin{equation}
T_{A,B}z =z -P_Az +P_BR_Az =(x, \rho) -(x, 0) +\delta(x, -\rho) =(\delta x, (1-\delta)\rho) \in X\times \RP.
\end{equation}
Hence, $(\forall\nnn)$ $z_n \in X\times \RP$. 
From $B_- =\menge{(x, \rho) \in X\times \RR}{-\sqrt{1 -\|x\|^2} \leq \rho \leq 0}$, we have
\begin{equation}
B_-\subseteq B' :=\epi f :=\menge{(x, \rho) \in X\times \RR}{f(x) \leq \rho},
\end{equation}
where $f\colon X\to \RR\colon x\mapsto -\sqrt{1 -\|x\|^2}$. 
Since $R_Az_n \in X\times \RM$, $P_BR_Az_n =P_{B_-}R_Az_n =P_{B'}R_Az_n$, and so 
\begin{equation}
(\forall\nnn)\quad z_{n+1} :=T_{A,B}z_n =T_{A,B'}z_n.
\end{equation}
According to \cref{f:P-E}, $(z_n)_\nnn$ converges finitely to a point in $A\cap B' =A\cap B$.

\emph{Case 2:} $0 <\theta <1$. 
Let $B' :=\epi f$, where $f\colon X\to \RR\colon x\mapsto \theta -\sqrt{1 -\|x\|^2}$.
Then $B \subseteq B'$, $A\cap B =A\cap B' =\menge{(x, 0) \in X\times \RR}{\theta -\sqrt{1 -\|x\|^2} \leq 0} \subseteq A =X\times \{0\}$, 
and $B'\smallsetminus B =\menge{(x, \rho) \in X\times \RR}{\theta +\sqrt{1 -\|x\|^2} <\rho} \subseteq X\times \RPP$,
which implies $(\forall c \in A\cap B)$ $d_{B'\smallsetminus B}(c) >0$.
Following \cref{l:id}\ref{l:id_sup},
$B$ and $B'$ are locally identical around $A\cap B$. 
By using \cref{t:global}, 
\begin{equation}
(\exists n_0 \in \NN)(\forall n \geq n_0)\quad z_{n+1} =T_{A,B'}z_n,
\end{equation}
and again by \cref{f:P-E}, we are done.
\end{proof}

\begin{remark}
It follows from \cref{ex:epi_infinite} that the conclusion of \cref{p:H-B} no longer holds without Slater's condition $A\cap \inte B \neq\varnothing$. 
\end{remark}

\begin{proposition}
\label{p:intersect}
Let $A =\bigcap_{i \in I} A_i$ and $B =\bigcap_{j \in J} B_j$ 
be finite intersections of closed convex sets in $X$ such that $A\cap B \neq\varnothing$.
Suppose that $(\forall x \in \Fix T_{A,B})(\exists i \in I)(\exists j \in J)$ 
both $(A, A_i)$ and $(B, B_j)$ are locally identical around $P_Ax$. 
Then the following holds for any DRA sequence $(x_n)_\nnn$ with respect to $(A, B)$:
\begin{equation}
\label{e:restrict_ij}
(\exists i \in I)(\exists j \in J)(\exists n_0 \in \NN)(\forall n \geq n_0)\quad x_{n+1} =T_{A_i,B_j}x_n.
\end{equation}
\end{proposition}
\begin{proof}
Since $A$ and $B$ are closed convex, \cref{f:cvg}\ref{f:cvg_FixT} gives 
$x_n \to x \in \Fix T_{A,B}$ with $P_Ax \in A\cap B$.
By assumption, $(\exists i \in I)(\exists j \in J)$ 
both $(A, A_i)$ and $(B, B_j)$ are locally identical around $P_Ax$.
Noting that $A\subseteq A_i$, $B\subseteq B_j$, the conclusion follows from \cref{l:expand}.
\end{proof}

\begin{corollary}
\label{c:int_intersect}
Let $A =\bigcap_{i \in I} A_i$ and $B =\bigcap_{j \in J} B_j$ 
be finite intersections of closed convex sets in $X$ such that $0 \in \inte(A -B)$.
Suppose that $(\forall x \in A\cap B)(\exists i \in I)(\exists j \in J)$ 
both $(A, A_i)$ and $(B, B_j)$ are locally identical around $x$. 
Then \eqref{e:restrict_ij} holds for any DRA sequence $(x_n)_\nnn$ with respect to $(A, B)$.
\end{corollary}
\begin{proof}
Since $0 \in \inte(A -B)$, \cref{f:cvg}\ref{f:cvg_AnB} implies $x_n \to x \in A\cap B$.
Then $P_Ax =x$, and \cref{p:intersect} completes the proof. 
\end{proof}

\begin{corollary}
\label{c:H-Bs}
Let $A$ be a hyperplane or a halfspace, and $B =\bigcap_{j \in J} B_j$ be a finite intersection of closed balls in $X$. 
Suppose that $A\cap \inte B \neq\varnothing$, and for all $x \in A\cap \bd B$, there exists a unique $j \in J$ such that $x \in \bd B_j$.
Then every DRA sequence $(x_n)_\nnn$ with respect to $(A, B)$ converges \emph{finitely} to a point in $A\cap B$.
\end{corollary}
\begin{proof}
From $A\cap \inte B \neq\varnothing$, we immediately have $0 \in \inte(A -B)$.
Let $x \in A\cap B$. If $x \in \inte B$, then $(\forall j \in J)$ $x \in \inte B_j$, and so $B$ and $B_j$ are locally identical around $x$,
following \cref{l:id}\ref{l:id_int}.
If $x \in \bd B$, then by assumption, there exists a unique $j \in J$ such that $x \in \bd B_j$, which implies that $B$ and $B_j$ are locally identical around $x$. Now using \cref{c:int_intersect},
\begin{equation}
(\exists j \in J)(\exists n_0 \in \NN)(\forall n \geq n_0)\quad x_{n+1} =T_{A,B_j}x_n.
\end{equation}   
Since $B\subseteq B_j$, we also have $A\cap \inte B_j \neq\varnothing$, and so $(x_n)_\nnn$ converges finitely due to \cref{p:H-B}. 
\end{proof}

\begin{figure}[!ht]
\centering
\includegraphics[width=0.9\columnwidth]{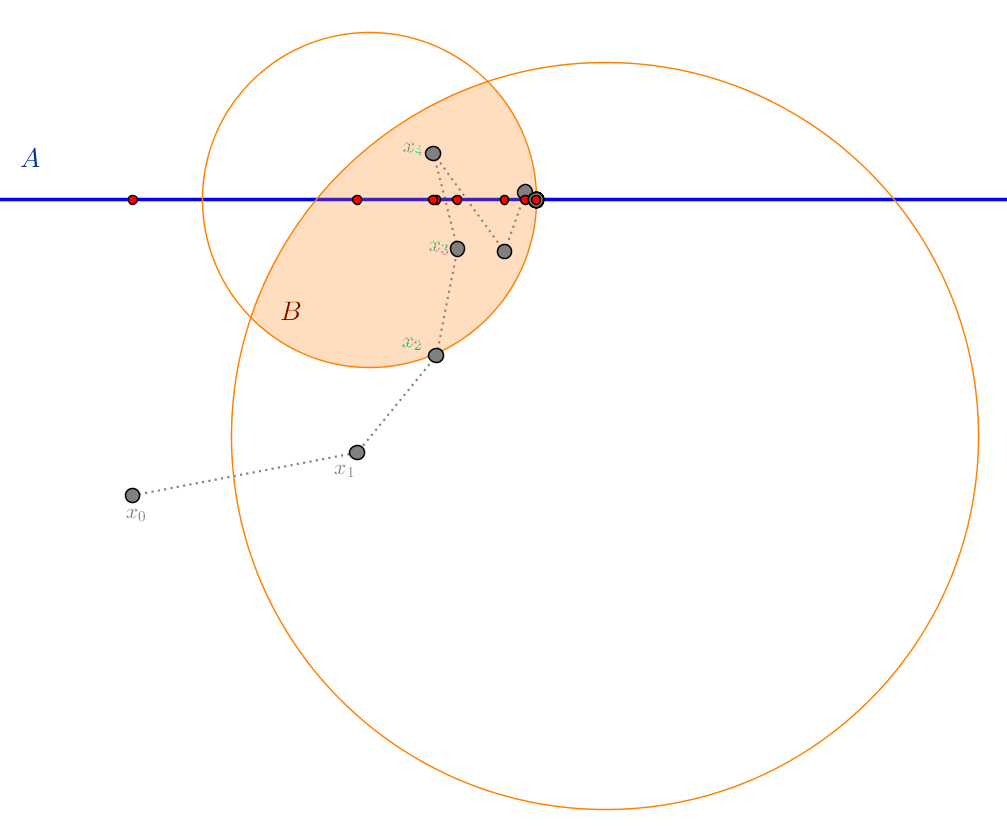}
\caption{A \texttt{GeoGebra} snapshot that illustrates \cref{c:H-Bs}.}\label{fg:H-Bs}
\end{figure}

\begin{corollary}
\label{c:P-B}
Let $A$ be a closed convex set, and $B$ be a closed ball in $\RR^2$ such that $A\cap \inte B \neq\varnothing$. 
Suppose that $A$ is locally identical with some polyhedral set around $A\cap \bd B$, 
and that no vertex of $A$ lies in $\bd B$.
Then every DRA sequence $(x_n)_\nnn$ with respect to $(A, B)$ converges \emph{finitely} to a point in $A\cap B$.
\end{corollary}
\begin{proof}
By \cref{t:finite}\ref{t:finite_AbdryB}\&\ref{t:finite_AintB}, 
it is sufficient to consider the case where $A$ is a polyhedral set in $\RR^2$ satisfying $A\cap \inte B \neq\varnothing$.
Then $0 \in \inte(A -B)$, and using \cref{f:cvg}\ref{f:cvg_AnB}, 
$x_n \to x \in A\cap B$, and this convergence is finite if $x \in A\cap \inte B$.
It thus suffices to consider the case where $x \in A\cap \bd B$.
We can write $A =\bigcap_{i=1}^m A_i$, where each $A_i$ is a halfplane in $\RR^2$.
Since all vertices of $A$ are not in $\bd B$, we deduce that $x$ is not a vertex of $A$.
Hence, $A$ and $A_i$ are locally identical around $x$ for some $i$.
Now using \cref{l:expand},
\begin{equation}
(\exists n_0 \in \NN)(\forall n \geq n_0)\quad x_{n+1} =T_{A_i,B}x_n.
\end{equation}   
Moreover, $A_i\cap \inte B \neq\varnothing$, 
and by \cref{p:H-B}, $(x_n)_\nnn$ converges finitely. 
\end{proof}

\section{Shrinking sets}

\label{s:shrink}

In this section we focus on cases where we use information of the
DRA for $(A,B)$ to understand the DRA for $(A',B')$ where
$A'\subseteq A$ and $B'\subseteq B$.

\begin{lemma}[Shrinking sets]
\label{l:restrict}
Let $A$ be a closed convex subset and $B$ be a closed (not necessarily convex) subset of $X$ such that $A\cap B \neq\varnothing$,
and let $x_0$ be in $X$.
Suppose that the DRA sequence $(x_n)_\nnn$ with respect to $(A, B)$, with starting point $x_0$, converges to $x \in X$.
Suppose further that there exist two closed sets $A'$ and $B'$ in $X$ such that
$A'\subseteq A$, $B'\subseteq B$, 
and that both $(A', A)$ and $(B', B)$ are locally identical around $c :=P_Ax \in A'\cap B'$.  
Then 
\begin{equation}
(\exists n_0 \in \NN)(\forall n \geq n_0)\quad x_{n+1} \in T_{A',B'}x_n.
\end{equation}
\end{lemma}
\begin{proof}
By assumption, there exists $\varepsilon \in \RPP$ such that 
\begin{equation}
A\cap \ball{c}{\varepsilon} =A'\cap \ball{c}{\varepsilon}
\quad\text{and}\quad
B\cap \ball{c}{\varepsilon} =B'\cap \ball{c}{\varepsilon}.
\end{equation}
Then, there is $n_0 \in \NN$ such that
\begin{equation}
(\forall n\geq n_0)\quad \|x_n -x\| <\varepsilon/3.
\end{equation}
Let $n \geq n_0$. Since $P_A$ is (firmly) nonexpansive (\cref{f:proj}\ref{f:proj_firm}),
\begin{equation}
\|P_Ax_n -c\| =\|P_Ax_n -P_Ax\| \leq \|x_n -x\| <\varepsilon/3,
\end{equation}
which implies $P_Ax_n \in \ball{c}{\varepsilon}$.
Using the convexity of $A$ and applying \cref{l:proj}\ref{l:proj_sub.a} for $A' \subseteq A$, we have 
$P_{A'}x_n =P_Ax_n$, and also $R_{A'}x_n =R_Ax_n$.
Noting that $x_{n+1} -x_n +P_Ax_n \in P_BR_Ax_n$ and
\begin{equation}
\|x_{n+1} -x_n +P_Ax_n -c\| \leq \|x_{n+1} -x\| +\|x_n -x\| +\|P_Ax_n -c\| <\varepsilon,
\end{equation}
we get $x_{n+1} -x_n +P_Ax_n \in P_BR_Ax_n\cap \ball{c}{\varepsilon}$, 
and then applying \cref{l:proj}\ref{l:proj_subs} for $B' \subseteq B$ yields
$x_{n+1} -x_n +P_Ax_n \in P_{B'}R_Ax_n =P_{B'}R_{A'}x_n$. 
Hence, $x_{n+1} \in x_n -P_{A'}x_n +P_{B'}R_{A'}x_n =T_{A',B'}x_n$.
\end{proof}

\begin{remark}
If $A'$ and $B'$ in \cref{l:restrict} are convex, 
then $T_{A',B'}$ is single-valued, and we have the conclusion that
\begin{equation}
\label{e:restrict}
(\exists n_0 \in \NN)(\forall n \geq n_0)\quad x_{n+1} =T_{A',B'}x_n,
\end{equation}
i.e., $(\exists n_0 \in \NN)(\forall\nnn)\quad T_{A,B}^nx_{n_0} =T_{A',B'}^nx_{n_0}$.
\end{remark}

\begin{corollary}
\label{c:restrict_j}
Let $A$ be a closed convex subset and $B =\bigcup_{j \in J} B_j$ be a finite union of disjoint closed convex sets in $X$ 
such that $A\cap B \neq\varnothing$, and let $x_0$ be in $X$.
Suppose that the DRA sequence $(x_n)_\nnn$ with respect to $(A, B)$, with starting point $x_0$, 
is bounded and \emph{asymptotically regular}, i.e., $x_n -x_{n+1} \to 0$.
Then $(x_n)_\nnn$ converges to a point $x \in \Fix T_{A,B}$, 
and there exists $j \in J$ such that 
\begin{equation}
P_Ax \in A\cap B_j \quad\text{and}\quad (\exists n_0 \in \NN)(\forall n \geq n_0)\quad x_{n+1} =T_{A,B_j}x_n.
\end{equation}
\end{corollary}
\begin{proof}
According to \cite[Theorem~2]{BN14}, $(x_n)_\nnn$ converges to a point $x \in \Fix T_{A,B}$. 
Since $A$ is convex, $P_Ax$ is a singleton, and by \eqref{e:xFixT'}, $P_Ax \in A\cap B$.
Then there exists $j \in J$ such that $P_Ax \in A\cap B_j$. 
By assumption, there exists $\varepsilon \in \RPP$ such that 
$(\forall k \in J\smallsetminus \{j\})$ $B_k \cap \ball{P_Ax}{\varepsilon} =\varnothing$.
This implies $B \cap \ball{P_Ax}{\varepsilon} =B_j \cap \ball{P_Ax}{\varepsilon}$, 
so $B$ and $B_j$ are locally identical around $P_Ax$. 
Now apply \cref{l:restrict}.
\end{proof}

\begin{corollary}
Let $A$ be a hyperplane or a halfspace, and $B =\bigcup_{j \in J} B_j$ be a finite union of disjoint closed balls in $X$ 
such that $A\cap B \neq\varnothing$, and $A\cap \inte B_j \neq\varnothing$ whenever $A\cap B_j \neq\varnothing$.
Let $x_0$ be in $X$.
Suppose that the DRA sequence $(x_n)_\nnn$ with respect to $(A, B)$, with starting point $x_0$, is bounded and \emph{asymptotically regular}, i.e., $x_n -x_{n+1} \to 0$.
Then $(x_n)_\nnn$ converges \emph{finitely} to a point $x \in A\cap B$.
\end{corollary}
\begin{proof}
Using \cref{c:restrict_j}, $x_n \to x \in \Fix T_{A,B}$, 
and there is $j \in J$ such that  
\begin{equation}
P_Ax \in A\cap B_j \quad\text{and}\quad (\exists n_0 \in \NN)(\forall n \geq n_0)\quad x_{n+1} =T_{A,B_j}x_n.
\end{equation}   
Then $A\cap B_j \neq\varnothing$, and by assumption, $A\cap \inte B_j \neq\varnothing$.
Now by \cref{p:H-B}, the convergence of $(x_n)_\nnn$ to $x$ is finite, and $x \in A\cap B_j \subseteq A\cap B$. 
\end{proof}

\section{When one set is finite}
\label{s:finiteset}

If the $B_j$ in \cref{c:restrict_j} are singletons and $A$ is either an affine subspace or a halfspace, then
it is possible to obtain stronger conclusions.
\begin{theorem}
Let $A$ be an affine subspace or a halfspace, and $B$ be a finite subset of $X$ such that $A\cap B \neq\varnothing$,
and let $x_0$ be in $X$. 
Suppose that the DRA sequence $(x_n)_\nnn$ with respect to $(A, B)$, with starting point $x_0$,  
is asymptotically regular, i.e., $x_n -x_{n+1} \to 0$.
Then $(x_n)_\nnn$ converges \emph{finitely} to a point $x \in \Fix T_{A,B}$ with $P_Ax \in A\cap B$.
\end{theorem}
\begin{proof}
Observe that $P_A$ is single-valued as $A$ is convex. 
According to \eqref{e:xFixT'}, it suffices to show that $x_n \to x \in \Fix T_{A,B}$ finitely.
Set
\begin{equation}
\label{e:bn}
(\forall\nnn)\quad b_n :=x_{n+1} -x_n +P_Ax_n \in P_BR_Ax_n \subseteq B.
\end{equation}
Let us first consider the case when $A$ is an affine subspace. 
Then we can represent $A =\menge{x \in X}{Lx =v}$, 
where $L$ is a linear operator from $X$ to a real Hilbert space $Y$, and $v \in \ran L$.
Denoting by $L^\dagger$ the Moore--Penrose inverse of $L$,
\cref{ex:proj_affine} gives 
\begin{equation}
\label{e:PAxn}
(\forall\nnn)\quad P_Ax_n =x_n -L^\dagger(Lx_n -v),
\end{equation}
and so
\begin{equation}
(\forall\nnn)\quad x_{n+1} =x_n -P_Ax_n +b_n =L^\dagger(Lx_n -v) +b_n.
\end{equation}
Since $L^\dagger LL^\dagger =L^\dagger$ (see \cite[Chapter II, Section 2]{Gro77}), we get 
\begin{subequations}
\label{e:LLxn+}
\begin{align}
(\forall\nnn)\quad L^\dagger(Lx_{n+1} -v) &=L^\dagger L(L^\dagger(Lx_n -v) +b_n) -L^\dagger v \\ 
&=L^\dagger(Lx_n -v) +L^\dagger(Lb_n -v),
\end{align}
\end{subequations}
and then \eqref{e:PAxn} gives
\begin{equation}
(\forall\nnn)\quad P_Ax_{n+1} =x_{n+1} -L^\dagger(Lx_{n+1} -v) =-L^\dagger(Lb_n -v) +b_n.
\end{equation}
Now in turn,
\begin{equation}
\label{e:xn++}
(\forall\nnn)\quad x_{n+2} =x_{n+1} -P_Ax_{n+1} +b_{n+1} =x_{n+1} +L^\dagger(Lb_n -v) -b_n +b_{n+1}.
\end{equation}
Using the asymptotic regularity of $(x_n)_\nnn$, \eqref{e:xn++} and \eqref{e:LLxn+} yield
\begin{subequations}
\begin{align}
\label{e:LLbn}
L^\dagger(Lb_n -v) &=L^\dagger L(x_{n+1} -x_n) \to 0, \\
b_{n+1} -b_n &=x_{n+2} -x_{n+1} -L^\dagger(Lb_n -v) \to 0.
\end{align}
\end{subequations}
Since $(b_n)_\nnn $ lies in $B$ and $B$ is finite, there exists $n_0 \in \NN$ such that
$(\forall n \geq n_0)$ $b_{n+1} =b_n =b \in B$.
Then by \eqref{e:LLbn}, $L^\dagger(Lb -v) =0$, 
which together with \eqref{e:xn++} gives
\begin{equation}
(\forall n \geq n_0)\quad x_{n+2} =x_{n+1} +L^\dagger(Lb -v) =x_{n+1},
\end{equation}
and $(x_n)_\nnn$ thus converges finitely.

Now consider the case when $A$ is a halfspace. 
Without loss of generality, we assume that 
$A =\menge{x \in X}{\scal{x}{u} \leq 0}$, where $u \in X$ and $\|u\| =1$.
Using \cref{ex:proj_h}\ref{ex:proj_half}, 
we have
\begin{equation}
(\forall\nnn)\quad P_Ax_n =\begin{cases}
x_n &\text{~if~} x_n \in A, \\
x_n -\scal{x_n}{u}u &\text{~if~} x_n \notin A, 
\end{cases}
\end{equation}
and by \eqref{e:bn}, 
\begin{equation}
\label{e:xn+}
(\forall\nnn)\quad x_{n+1} =\begin{cases}
b_n &\text{~if~} x_n \in A, \\
\scal{x_n}{u}u +b_n &\text{~if~} x_n \notin A.
\end{cases} 
\end{equation}
If $(\exists\nnn)$ $x_n \in A$ and $b_n \in A$, 
then \eqref{e:xn+} gives $x_{n+1} =b_n \in A\cap B$, and we are done.
Assume that $(\forall\nnn)$ $x_n \notin A$ or $b_n \notin A$. 
By using \eqref{e:xn+}, $(\forall\nnn)$ $x_n \in A$ $\Rightarrow$ $x_{n+1} =b_n \notin A$.  
Thus, the set $\menge{\nnn}{x_n \notin A}$ is infinite, 
and denoted by $(n_k)_{k \in \NN}$ the enumeration of that set, we have
\begin{equation}
(\forall k \in \NN)\quad x_{n_k} \notin A, \text{~i.e.,~}  \scal{x_{n_k}}{u} >0, 
\quad\text{and}\quad n_{k+1} -n_{k} \in \{1, 2\}.
\end{equation}
Then $x_{n_{k+1}} -x_{n_k} =x_{n_k+1} -x_{n_k}$ or 
$x_{n_{k+1}} -x_{n_k} =(x_{n_k+2} -x_{n_k+1}) +(x_{n_k+1} -x_{n_k})$, 
and the asymptotic regularity of $(x_n)_\nnn$ implies the one of $(x_{n_k})_{k \in \NN}$ and also of $(x_{n_k+1})_{k \in \NN}$.
Since $x_{n_k} \notin A$, \eqref{e:xn+} gives
\begin{equation}
\label{e:xnk}
x_{n_k+1} =\scal{x_{n_k}}{u}u +b_{n_k},
\end{equation}
and so
\begin{equation}
b_{n_{k+1}} -b_{n_k} =(x_{n_{k+1}+1} -x_{n_k+1}) -\scal{x_{n_{k+1}} -x_{n_k}}{u}u \to 0.
\end{equation} 
But $(b_{n_k})_{k \in \NN}$ is in the finite set $B$, there exists $k_0 \in \NN$ such that 
\begin{equation}
\label{e:bnk0}
(\forall k \geq k_0)\quad b_{n_{k+1}} =b_{n_k} =: b \in B.
\end{equation}
On the other hand, \eqref{e:xnk} implies
\begin{equation}
\label{e:xnk,u}
(\forall k \in \NN)\quad \scal{x_{n_k+1}}{u} =\scal{x_{n_k}}{u} +\scal{b_{n_k}}{u},
\end{equation}
and then 
\begin{equation}
\scal{b_{n_k}}{u} =\scal{x_{n_k+1} -x_{n_k}}{u} \to 0,
\end{equation}
which yields $\scal{b}{u} =0$, and thus $b \in A\cap B$.
Let $k \geq k_0$. It follows from \eqref{e:bnk0} and \eqref{e:xnk,u} that 
\begin{equation}
\scal{x_{n_k+1}}{u} =\scal{x_{n_k}}{u} +\scal{b}{u} =\scal{x_{n_k}}{u}.
\end{equation}
Hence $x_{n_k+1} \notin A$ as $x_{n_k} \notin A$. 
We obtain $n_{k+1} =n_k +1$, and by combining with \eqref{e:xnk} and \eqref{e:bnk0},
\begin{equation}
x_{n_k+2} =\scal{x_{n_k+1}}{u}u +b =\scal{x_{n_k}}{u} +b =x_{n_k+1},
\end{equation}
which completes the proof.
\end{proof}

The following examples illustrate that without asymptotic regularity 
a DRA sequence with respect to $(A, B)$ may fail to converge.
\begin{example}
Suppose that $X =\RR^2$, $A =\RR\times \{0\}$ and $B =\{(0, -2), (1, 2), (-2, 0)\}$. 
Then $A\cap B \neq\varnothing$ but the DRA sequence with respect to $(A, B)$ with starting point $x_0 =(0, -1)$ 
does not converge since it cycles between two points $x_0 =(0, -1)$ and $x_1 =(1, 1)$.
\end{example}

\begin{example}
Suppose that $X =\RR^2$, that $A =\RR\times \RR_-$ is a halfspace, 
and that $B =\{(2, 5), (20, -20), (8, 7), (-20, 0)\}$ is a finite set. 
Then $A\cap B \neq\varnothing$ but when started at $x_0 =(2, 17)$, 
the DRA cycles between four points $x_0 =(2, 17)$, $x_1 =(20, -3)$, $x_2 =(8, 7)$ and $x_3 =(2, 12)$, 
as shown in \cref{fg:H-Ps} which was created by \texttt{GeoGebra} \cite{GGB}.
\end{example}

\begin{figure}[!ht]
\centering
\includegraphics[width=0.9\columnwidth]{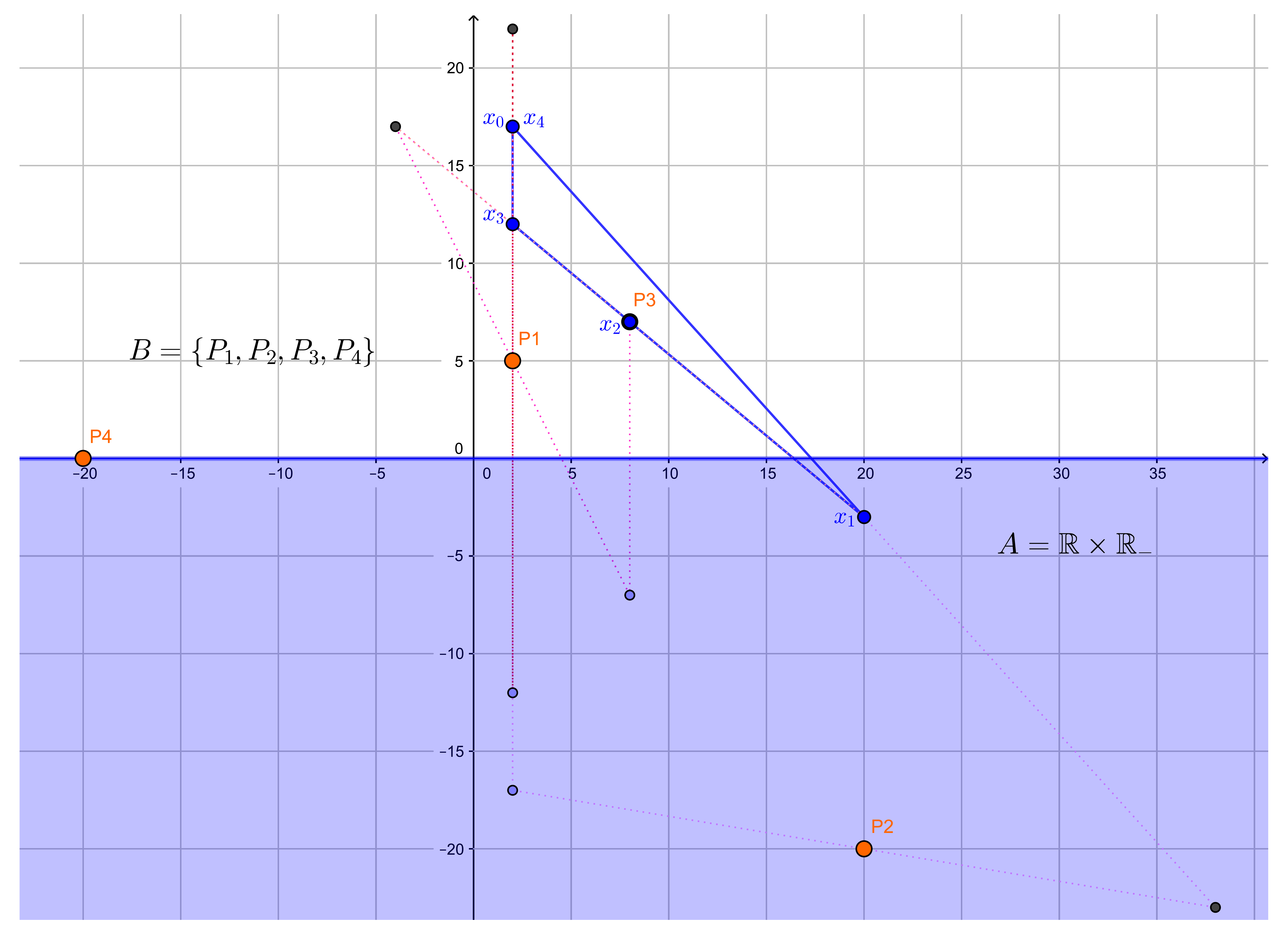}
\caption{A 4-cycle of the DRA for a halfspace and a finite set.}\label{fg:H-Ps}
\end{figure}

\begin{remark}[Order matters]
Notice that if $A$ is a halfspace and $B$ is a finite subset of $X$ such that $A\cap B \neq\varnothing$, 
then every DRA sequence with respect to \emph{$(B, A)$} converges finitely due to \cite[Theorem~4.2]{ABT15}.
Recall from \cite{BM15} that if we work with an affine subspace instead of a halfspace, 
then the quality of convergence of the DRA sequence with respect to $(A, B)$ is the same as the one with respect to $(B, A)$.
\end{remark}

\begin{theorem}
Let $A$ be either a hyperplane or a halfspace of $X$, 
and $B$ be a finite subset of one in two halfspaces generated by $A$, 
and let $x_0$ be in $X$.
Then either: 
(i) the DRA sequence $(x_n)_\nnn$ with respect to $(A, B)$, with starting point $x_0$,  
converges \emph{finitely} to a point $x \in \Fix T_{A,B}$ with $P_Ax \in A\cap B$, 
or (ii) $A\cap B =\varnothing$ and $\|x_n\| \to +\infty$ 
in which case $(P_Ax_n)_\nnn$ converges \emph{finitely} to a best approximation solution $a \in A$ relative to $A$ and $B$
in the sense that $d_B(a) =\min_{a' \in A} d_B(a')$.
\end{theorem}
\begin{proof}
\emph{Case 1:} $A$ is a hyperplane. 
Without loss of generality, we assume that 
\begin{subequations}
\begin{equation}
A =H :=\menge{x \in X}{\scal{x}{u} =0} \quad\text{with}\quad u \in X, \|u\| =1,
\end{equation}
and that 
\begin{equation}
\label{e:B>=0}
(\forall b \in B)\quad \scal{b}{u} \geq 0.
\end{equation}
\end{subequations}
By \cref{ex:proj_h}\ref{ex:proj_hyper},
\begin{equation}
(\forall x \in X)\quad P_Ax =x -\scal{x}{u}u.
\end{equation} 
Therefore,
\begin{equation}
(\forall x \in X)\quad R_Ax =2P_Ax -x =x -2\scal{x}{u}u,
\end{equation}
and also
\begin{equation}
(\forall x \in X)\quad d_A(x) =\|x -P_Ax\| =|\scal{x}{u}|.
\end{equation}
Now setting
\begin{equation}
(\forall\nnn)\quad b_n :=x_{n+1} -x_n +P_Ax_n \in P_BR_Ax_n \subseteq B,
\end{equation}
we have
\begin{subequations}
\begin{align}
\label{e:x+}
(\forall\nnn)\qquad x_{n+1} &=T_{A,B}x_n =x_n -P_Ax_n +P_BR_Ax_n =\scal{x_n}{u}u +b_n, \\
\label{e:x+,u}
\scal{x_{n+1}}{u} &=\scal{\scal{x_n}{u}u +b_n}{u} =\scal{x_n}{u} +\scal{b_n}{u} \geq \scal{x_n}{u}, \\
\label{e:PAx+}
P_Ax_{n+1} &=x_{n+1} -\scal{x_{n+1}}{u}u =b_n -\scal{b_n}{u}u, \\
\label{e:RAx+}
R_Ax_{n+1} &=x_{n+1} -2\scal{x_{n+1}}{u}u =b_n -(\scal{x_n}{u} +2\scal{b_n}{u})u,
\end{align}
\end{subequations}
and so
\begin{equation}
\label{e:x++}
(\forall\nnn)\quad x_{n+2} =(\scal{x_n}{u} +\scal{b_n}{u})u +b_{n+1} =x_{n+1} +\scal{b_n}{u}u +b_{n+1} -b_n.
\end{equation}
It follows that $b_n -R_Ax_{n+1} =(\scal{x_n}{u} +2\scal{b_n}{u})u$, and 
\begin{subequations}
\begin{align}
&\|b_{n+1} -R_Ax_{n+1}\|^2 =\|(b_{n+1} -b_n) +(b_n -R_Ax_{n+1})\|^2 \\ 
&=\|b_{n+1} -b_n\|^2 +2(\scal{x_n}{u} +2\scal{b_n}{u})\scal{b_{n+1} -b_n}{u} +\|b_n -R_Ax_{n+1}\|^2. 
\end{align}
\end{subequations}
From $b_{n+1} =P_BR_Ax_{n+1}$ and $b_n \in B$, we have
$\|b_{n+1} -R_Ax_{n+1}\| \leq \|b_n -R_Ax_{n+1}\|$, which yields
\begin{subequations}
\label{e:bns}
\begin{align}
0 \leq \|b_{n+1} -b_n\|^2 &\leq 2(\scal{x_n}{u} +2\scal{b_n}{u})\scal{b_n -b_{n+1}}{u} \\ 
&=2(\scal{x_n}{u} +2\scal{b_n}{u})(\scal{b_n}{u} -\scal{b_{n+1}}{u}).
\end{align}
\end{subequations}

\emph{Case 1.1:} $(\forall\nnn)$ $\scal{x_n}{u} \leq 0$. 
By combining with \eqref{e:x+,u}, the sequence $(\scal{x_n}{u})_\nnn$ converges, and so
\begin{equation}
\scal{b_n}{u} =\scal{x_{n+1}}{u} -\scal{x_n}{u} \to 0.
\end{equation}
But $(b_n)_\nnn$ lies in the finite set $B$; 
hence, there exists $n_0 \in \NN$ such that $(\forall n\geq n_0)$ $\scal{b_n}{u} =0$, equivalently, $b_n \in A$.
Then \eqref{e:bns} implies $(\forall n\geq n_0)$ $b_{n+1} =b_n$, and by \eqref{e:x++}, $x_{n+2} =x_{n+1} \in \Fix T_{A,B}$.

\emph{Case 1.2:} $(\exists n_0 \in \NN)$ $\scal{x_{n_0}}{u} >0$. 
Then \eqref{e:x+,u} and \eqref{e:B>=0} give
\begin{equation}
\label{e:xubu}
(\forall n\geq n_0)\quad \scal{x_n}{u} +2\scal{b_n}{u} >0.
\end{equation} 
Combining with \eqref{e:bns}, this implies
\begin{equation}
\label{e:bn_decrease}
(\forall\nnn)\quad 0 \leq \scal{b_{n+1}}{u} \leq \scal{b_n}{u},
\end{equation}
and the sequence $(\scal{b_n}{u})_\nnn \subseteq B$ thus converges. 
Since again $B$ is finite, there exists $n_1 \in \NN$, $n_1 \geq n_0$ such that 
$(\forall n\geq n_1)$ $\scal{b_{n+1}}{u} =\scal{b_n}{u}$, 
which yields $b_{n+1} =b_{n} =: b \in B$ due to \eqref{e:bns}.
By combining with \eqref{e:PAx+}, 
\begin{equation}
\label{e:PAx+,b}
(\forall n\geq n_1)\quad P_Ax_{n+1} =b -\scal{b}{u}u \quad\text{and}\quad \|P_Ax_{n+1} -b\| =|\scal{b}{u}| =\scal{b}{u},
\end{equation}
so $(P_Ax_n)_\nnn$ converges finitely.
Furthermore, if $\scal{b}{u} =0$, i.e., $b \in A$, then $b \in A\cap B$, 
in which case $A\cap B \neq\varnothing$ and by \eqref{e:x++},
$(\forall n\geq n_1)$ $x_{n+2} =x_{n+1} \in \Fix T_{A,B}$.

Now assume that $\scal{b}{u} \neq 0$. Then $\scal{b}{u} >0$ due to \eqref{e:B>=0}.
It follows from \eqref{e:x+,u} and \eqref{e:RAx+} that
\begin{equation}
(\forall n \geq n_1)\quad R_Ax_{n+1} =b -(\scal{x_{n_1}}{u} +(n -n_1 +2)\scal{b}{u})u.
\end{equation}
Let $n \geq n_1$, and let $b' \in B$. 
Since $b =b_{n+1} =P_BR_Ax_{n+1}$, we have $\|b -R_Ax_{n+1}\| \leq \|b' -R_Ax_{n+1}\|$,
and so
\begin{equation}
\|b -R_Ax_{n+1}\|^2 \leq \|b' -b\|^2 +2\scal{b' -b}{b -R_Ax_{n+1}} +\|b -R_Ax_{n+1}\|^2,
\end{equation}
which implies
\begin{subequations}
\begin{align}
\|b' -b\|^2 &\geq 2\scal{b -b'}{(\scal{x_{n_1}}{u} +(n -n_1 +2)\scal{b}{u})u} \\
&=2(\scal{x_{n_1}}{u} +(n -n_1 +2)\scal{b}{u})(\scal{b}{u} -\scal{b'}{u}).
\end{align}
\end{subequations}
Noting that $\scal{x_{n_1}}{u} +(n -n_1 +2)\scal{b}{u} \to +\infty$, we deduce $\scal{b}{u} \leq \scal{b'}{u}$. 
Hence 
\begin{equation}
0 <\scal{b}{u} =\min_{b' \in B} \scal{b'}{u} =\min_{b' \in B} d_A(b').
\end{equation}
This yields $A\cap B =\varnothing$, and by \eqref{e:x+,u},
\begin{equation}
\|x_n\| \geq \scal{x_n}{u} =\scal{x_{n_1}}{u} +(n -n_1)\scal{b}{u} \to +\infty \quad\text{as}\quad n \to +\infty,
\end{equation}
while by \eqref{e:PAx+,b}, $(\forall n\geq n_1)$ $(P_Ax_{n+1}, b)$ is a best approximation pair relative to $A$ and $B$.

\emph{Case 2:} $A$ is a halfspace. By assumption, we assume without loss of generality that either
\begin{subequations}
\begin{equation}
\label{e:A=H+}
A =H_+ :=\menge{x \in X}{\scal{x}{u} \geq 0} \quad\text{and}\quad B \subseteq H_+,
\end{equation}
or
\begin{equation}
\label{e:A=H-}
A =H_- :=\menge{x \in X}{\scal{x}{u} \leq 0} \quad\text{and}\quad B \subseteq H_+,
\end{equation}
\end{subequations}
where $u \in X$ and $\|u\| =1$.

\emph{Case 2.1:} \eqref{e:A=H+} holds. 
If $(\forall\nnn)$ $\scal{x_n}{u} \leq 0$, i.e. $x_n \in H_-$, then $P_Ax_n =P_Hx_n$, so 
\begin{equation}
x_{n+1} =T_{A,B}x_n =T_{H,B}x_n,
\end{equation}
and according to \emph{Case 1.1}, 
we must have $H\cap B \neq\varnothing$
and the finite convergence of $(x_n)_\nnn$. 
If $(\exists n_0 \in \NN)$ $\scal{x_{n_0}}{u} \geq 0$, i.e. $x_{n_0} \in H_+$, then $R_Ax_{n_0} =P_Ax_{n_0} =x_{n_0}$,
which yields $x_{n_0+1} =x_{n_0} -P_Ax_{n_0} +P_BR_Ax_{n_0} =P_Bx_{n_0} \in B =A\cap B$, and we are done.

\emph{Case 2.2:} \eqref{e:A=H-} holds. 
If $\scal{x_0}{u} \leq 0$, i.e. $x_0 \in H_-$, then $R_Ax_0 =P_Ax_0 =x_0$, and thus
$x_1 =x_0 -P_Ax_0 +P_BR_Ax_0 =P_Bx_0 \in B \subseteq H_+$.
It is therefore sufficient to consider $\scal{x_0}{u} \geq 0$, i.e. $x_0 \in H_+$. 
Then $P_Ax_0 =P_Hx_0$, $x_1 =T_{A,B}x_0 =T_{H,B}x_0$, 
and by \eqref{e:x+,u}, $\scal{x_1}{u} \geq \scal{x_0}{u} \geq 0$. 
This yields 
\begin{equation}
(\forall\nnn)\quad x_n \in H_+ \quad\text{and}\quad x_{n+1} =T_{H,B}x_n.
\end{equation}
Now apply Case~1.
\end{proof}

\begin{example}
\label{ex:H-2Ps}
Suppose that $X =\RR^2$, $A =\RR\times \{0\}$ and $B =\{(0, 1), (1, 2)\}$. 
Then $A\cap B =\varnothing$, and for starting point $x_0 \in \left]1, +\infty\right[\times \{-1\}$, 
the DRA sequence $(x_n)_\nnn$ with respect to $(A, B)$ satisfies 
$(\forall n \in \{2, 3, \dots\})$ $x_n =(0, n)$ and $P_Ax_n =(0, 0)$. 
See \cref{fg:H-2Ps} for an illustration, created with \texttt{GeoGebra} \cite{GGB}.
\end{example}

\begin{figure}[!ht]
\centering
\includegraphics[width=0.6\columnwidth]{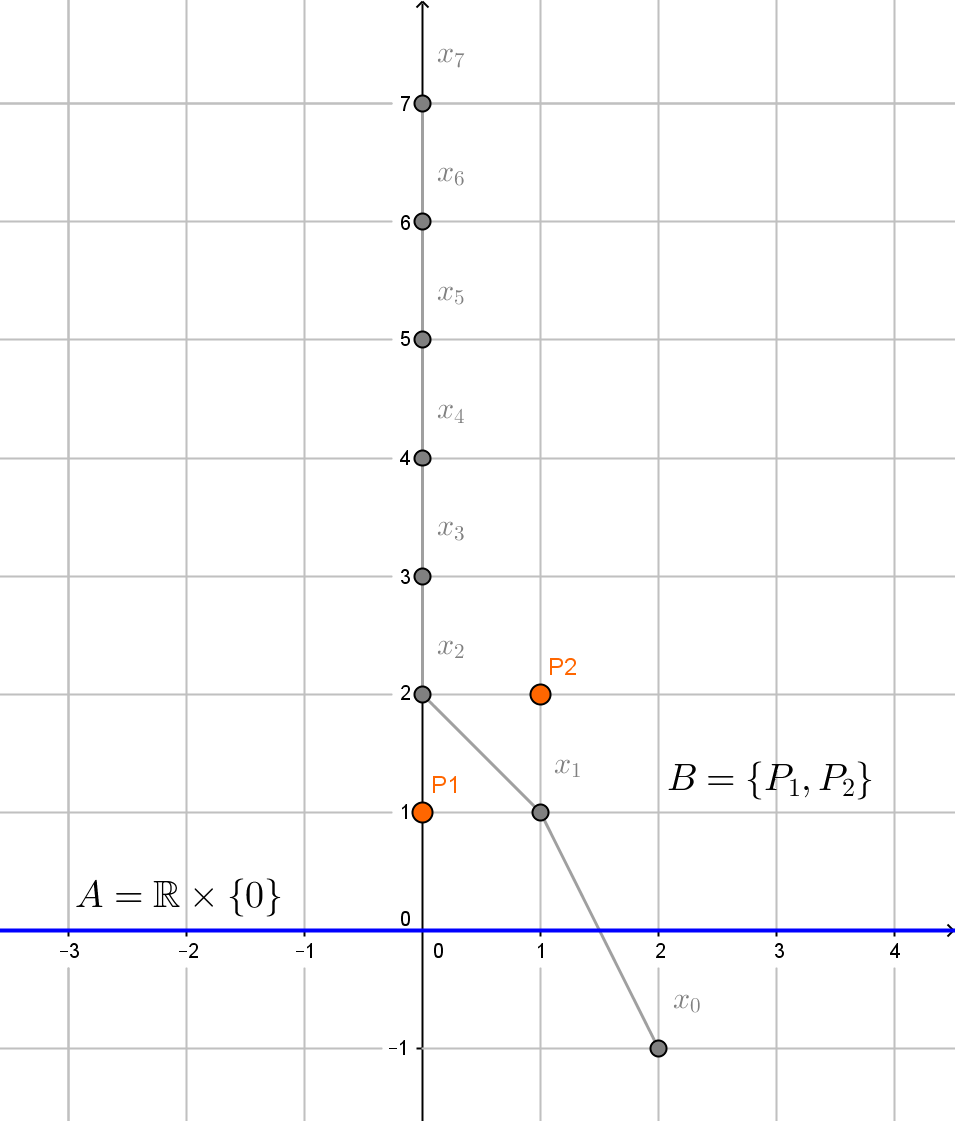}
\caption{An illustration for \cref{ex:H-2Ps} with the starting point $x_0 =(2, -1)$.}\label{fg:H-2Ps}
\end{figure}

\section{When $A$ is an affine subspace and $B$ is a polyhedron}
\label{s:8}

In view of \cref{d:local}, we recall a result on finite convergence of the Douglas--Rachford algorithm under Slater's condition. 

\begin{fact}[Finite convergence of DRA in the affine-polyhedral case]
\label{f:A-P}
Let $A$ be an affine subspace and $B$ be a closed convex subset of $X$ such that \emph{Slater's condition}
\begin{equation}
\label{e:AintB}
A\cap \inte B \neq\varnothing
\end{equation}
holds. Suppose that $B$ is locally identical with some polyhedral set around $A\cap \bd B$.
Then every DRA sequence $(x_n)_\nnn$ with respect to $(A, B)$ converges \emph{finitely} to a point in $A\cap B$.
\end{fact}
\begin{proof}
Combine \cite[Theorem~3.7 and Definition~2.7]{BDNP15} with \cref{d:local}.
\end{proof}

A natural question is whether the conclusion of \cref{f:A-P} holds 
when the Slater's condition $A\cap \inte B \neq\varnothing$ is replaced by $A\cap B \neq\varnothing$ and $\inte B \neq\varnothing$. 
In the sequel, we shall provide a positive answer in $\RR^2$ (\cref{t:poly2D}) and a negative answer in $\RR^3$ (\cref{ex:poly3D}).
For the next little while, we work with 
\begin{empheq}[box=\mybluebox]{equation}
X =\RR^2 \quad\text{and}\quad A =\RR\times \{0\},
\end{empheq}
and consider the (counter-clockwise) rotator defined by
\begin{equation}
(\forall\theta \in \RR)\quad \mathcal R_\theta :=
\begin{bmatrix}
\cos\theta & -\sin\theta \\ 
\sin\theta & \cos\theta 
\end{bmatrix}.
\end{equation}
Let $\theta \in [0, \pi]$, and set
\begin{equation}
e_0 :=(1, 0), \quad e_{\pi/2} :=(0, 1), \quad e_\theta :=(\cos\theta)e_0 +(\sin\theta)e_{\pi/2}.
\end{equation}
Then $\RP\times \{0\} =\RP\cdot e_0$ is the positive $x$-axis, 
and $\mathcal R_\theta(\RP\times \{0\}) =\RP\cdot e_\theta$
is the ray starting at $0 \in X$ and making an angle of $\theta$ 
with respect to $\RP\times \{0\}$ in counter-clockwise direction.

For $x, y \in X$, we write $\angle(x, y) :=\theta$ if $y \in \RP\mathcal R_\theta(x)$,
and $\angle(x, y) =\theta -\pi$ if $y \in \RM\mathcal R_\theta(x)$.

\begin{fact}
\label{f:2lines}
Let $\theta \in [0, \pi]$. Then
\begin{equation}
T_{A,\mathcal R_\theta(A)} =\Id -P_A +P_{\mathcal R_\theta(A)}R_A =(\cos\theta)\mathcal R_\theta.
\end{equation}
\end{fact}
\begin{proof}
This follows from \cite[Section~5]{BBNPW14}.
\end{proof}

\begin{lemma}
\label{l:raystep}
Assume that $\theta \in [0, \pi]$, $B =\mathcal R_\theta(\RP\times \{0\})$, $H =B^\oplus$, and $H' =R_A(H)$. 
Let $x =(\alpha, \beta) \in X$, and set $x_+ =T_{A,B}x$. 
Then $x_+ =(0, \beta)$ if $x \not\in H'$, and $x_+ =(\cos\theta)\mathcal R_\theta(z)$ otherwise.
In the latter case, $x_+ =0$ if $\theta =\pi/2$, and
\begin{equation}
\label{e:raystep}
\angle(x, x_+) =
\begin{cases}
\theta, &\text{ if } \theta <\pi/2; \\ 
\theta -\pi, &\text{ if } \theta >\pi/2. 
\end{cases}
\end{equation}
Furthermore, 
\begin{equation}
\Fix T_{A,B} =
\begin{cases}
\RP\times \RR, &\text{ if } \theta =0; \\ 
\{0\}\times \RP, &\text{ if } 0 <\theta <\pi; \\ 
\RM\times \RR, &\text{ if } \theta =\pi. 
\end{cases}
\end{equation}
\end{lemma}
\begin{proof}
We have $P_Ax =(\alpha, 0)$ and $R_Ax =(\alpha, -\beta)$.
If $x =(\alpha, \beta) \not\in H'$, then $R_Ax \not\in H$, and so $P_BR_Ax =(0, 0)$, which yields
\begin{equation}
x_+ =(\Id -P_A +P_BR_A)x =(\alpha, \beta) -(\alpha, 0) +(0, 0) =(0, \beta).
\end{equation} 
Now we consider the case $x \in H'$. 
Then $R_Ax \in H$, so $P_BR_Ax =P_{\mathcal R_\theta(A)}R_Ax$, 
and by applying \cref{f:2lines}, 
\begin{equation}
x_+ =(\cos\theta)\mathcal R_\theta(x).
\end{equation}
The rest is clear.
\end{proof}

\begin{lemma}
\label{l:ray}
Let 
\begin{equation}
A =\RR\times \{0\} \quad\text{and}\quad B =\mathcal R_\theta(\RP\times \{0\}),
\end{equation}
where $\theta \in [0, \pi]$.
Then every DRA sequence $(x_n)_\nnn$ with respect to $(A, B)$ converges to a point $x \in \Fix T_{A,B}$,
and the ``shadow sequence'' $(P_Ax_n)_\nnn$ converges to $P_Ax \in A\cap B$ 
in at most $N$ iterations, where
\begin{equation}
N =
\begin{cases}
\lfloor \frac{\pi}{\theta} \rfloor +3, &\text{ if } \theta \leq \pi/2; \\ 
\lfloor \frac{\pi}{\pi -\theta} \rfloor +3, &\text{ if } \theta > \pi/2. 
\end{cases}
\end{equation}
\end{lemma}
\begin{proof}
Set $H =B^\oplus$, and $H' =R_A(H)$.
We will study the behavior of the iterations in regions
\begin{subequations}
\begin{align}
R_1 &=\menge{(\alpha, \beta) \in X}{(\alpha, \beta) \not\in H',\; \beta <0}, \\ 
R_2 &=H', \\
R_3 &=\menge{(\alpha, \beta) \in X}{(\alpha, \beta) \not\in H',\; \beta \geq 0}
\end{align}
\end{subequations}
as shown in Figure~\ref{fg:ray}.
\begin{figure}[!ht]
\centering
\includegraphics[width =0.48\columnwidth]{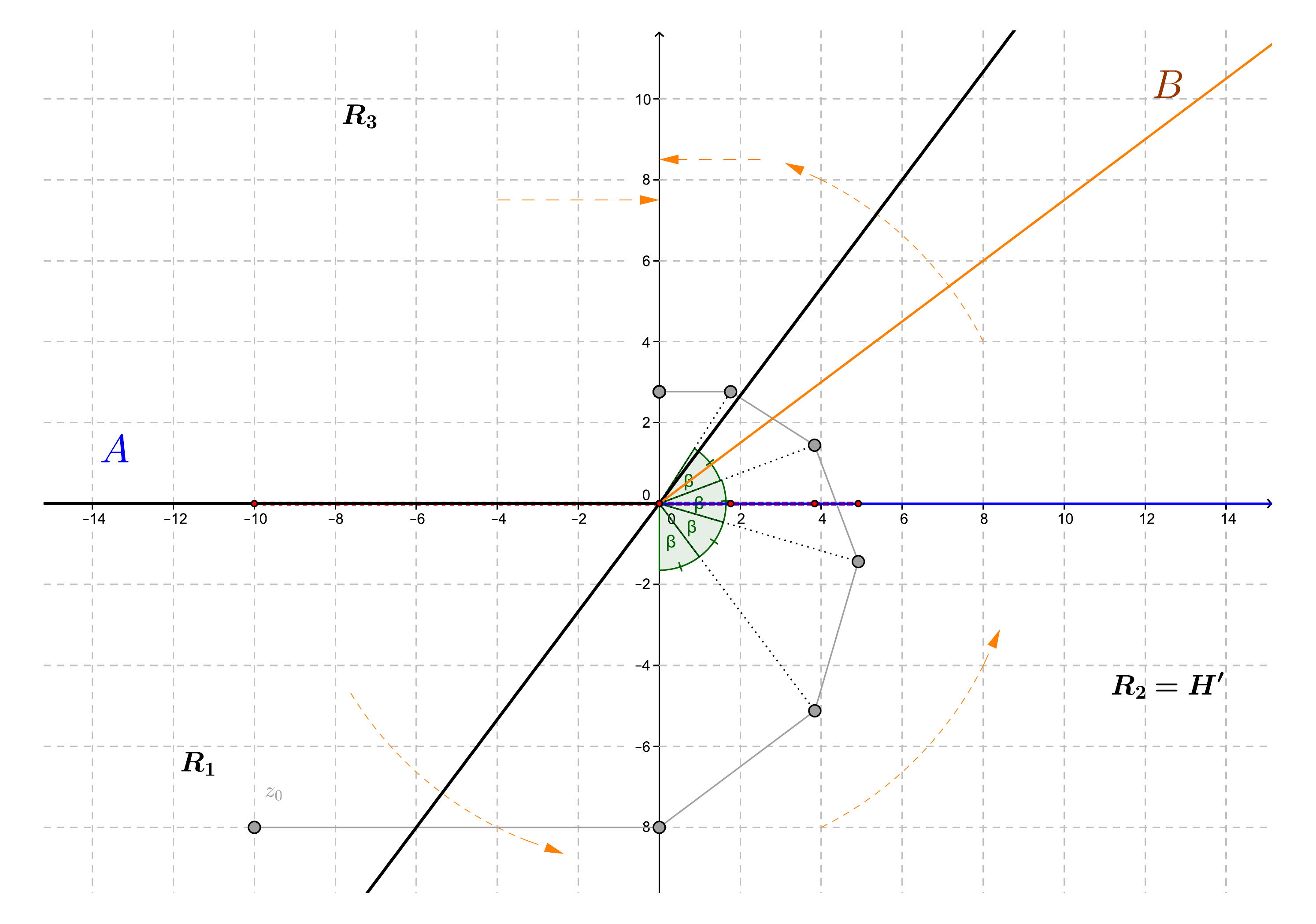}
\includegraphics[width =0.48\columnwidth]{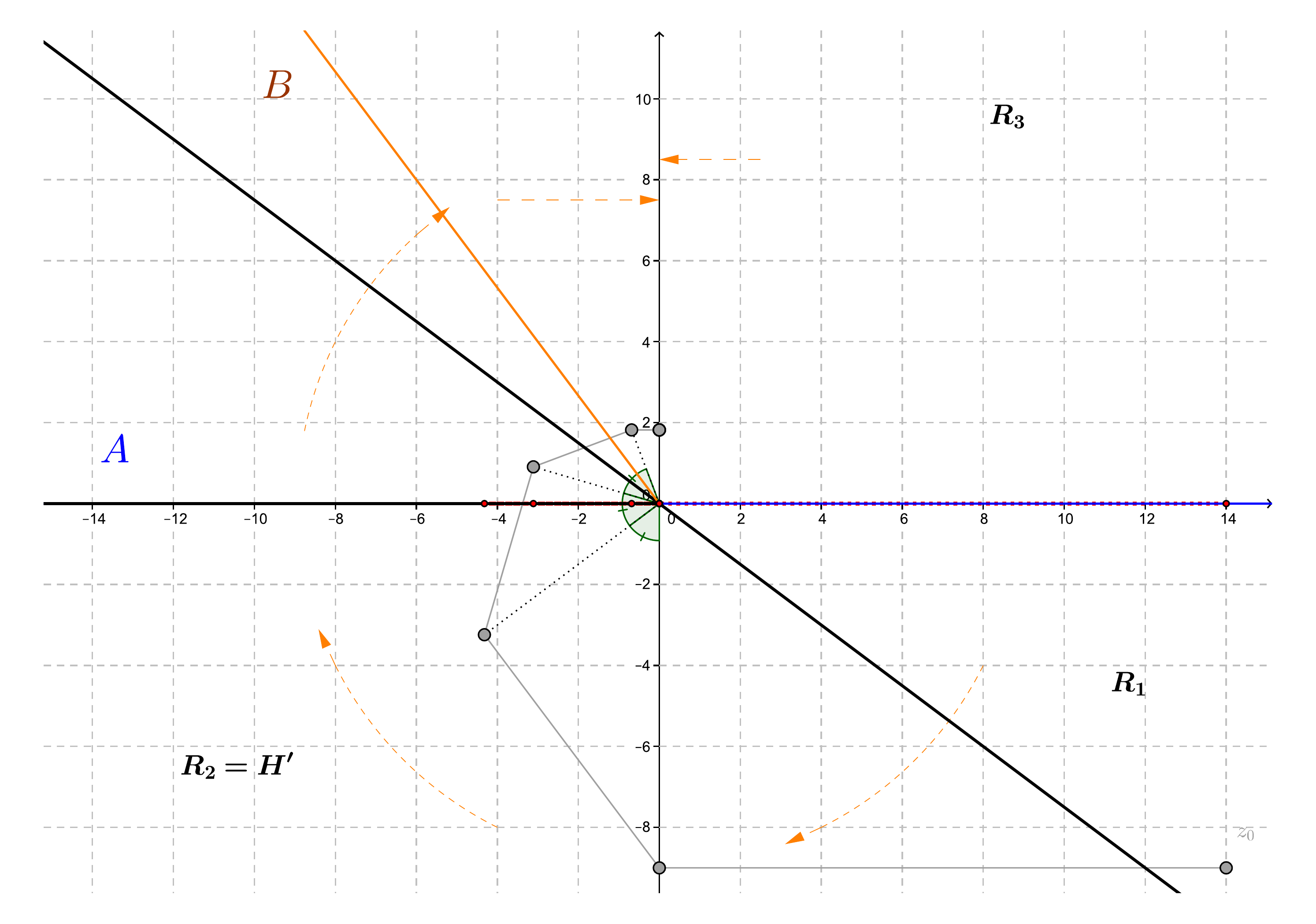}
\caption{The DRA for the case of a line and a ray in the Euclidean plane}
\label{fg:ray}
\end{figure}

Since $\theta \in [0, \pi]$, we have $0\times \RP \subseteq H$, and so $\{0\}\times \RM \subseteq H'$.
Set $x_0 :=(\alpha_0, \beta_0) \in X$.
According to \cref{l:raystep}, if $x_0 \in R_1$, then $x_1 =(0, \beta_0) \in \{0\}\times \RM \subseteq H'$;
if $x_0 \in R_3$, then $x_1 =(0, \beta_0) \in 0\times \RP \subseteq\Fix T_{A,B}$. 
So it is sufficient to consider the case $x_0 \in H' =R_2$.  
If $\theta =\pi/2$, we have immediately $x_1 =0 \in A\cap B$. 
Now we assume without loss of generality that $\theta <\pi/2$. 
Then, \eqref{e:raystep} yields the implication
\begin{equation}
x_0, \dots, x_{n-1} \in R_2 \quad\Rightarrow\quad \angle(x_0, x_n) =n\theta.
\end{equation}
There thus exists $n_0 \in \NN$ such that
\begin{equation}
x_0, \dots, x_{n_0-1} \in R_2, \quad\text{and}\quad x_{n_0} \not\in R_2,
\end{equation} 
which yields $x_{n_0} \in R_3$. 
Using again \cref{l:raystep}, $x_{n_0+1} =(0, \beta_{n_0}) \in 0\times \RP \subseteq\Fix T$.
Noting that 
\begin{equation}
\angle(x_0, x_{n_0}) =n_0\theta \leq \pi +\theta,
\end{equation}
we get $n_0 \leq \lfloor \pi/\theta \rfloor +1$.
Hence, $x_n =x \in \Fix T_{A,B}$ and $P_Ax_n =P_Ax \in A\cap B$ for all $n \geq \lfloor \pi/\theta \rfloor +3$ iterations.
\end{proof}

\begin{lemma}
\label{l:cone}
Let either $A =\RR\times \{0\}$ or $A =\RR\times \RM$, and let $B$ be the convex cone generated by the union of the rays 
\begin{equation}
B_1 =\mathcal R_{\theta_1}(\RP\times \{0\}) \quad\text{and}\quad B_2 =\mathcal R_{\theta_2}(\RP\times \{0\})
\end{equation}
with $\theta_1, \theta_2 \in [0, \pi]$.
Then the DRA applied to $(A, B)$ converges \emph{finitely globally uniformly}
in the sense that there exists $N \in \NN$ such that $(\forall x \in X)$ the sequence $(T_{A,B}^nx)_\nnn$ 
converges to a point in $\Fix T_{A,B}$ in at most $N$ iterations. 
\end{lemma}
\begin{proof}
We shall prove this for the case $A =\RR\times \{0\}$, the other case being similar.
For $i \in \{1, 2\}$, set $H_i =B_i^\oplus$, $H_i' =R_A(H_i)$, $B_i' =R_A(B_i)$, 
and let $B_1'' =\mathcal R_{\pi/2}(B_1')$, $B_2'' =\mathcal R_{\pi/2}^{-1}(B_2')$.
Without loss of generality, we distinguish two cases: 
$0 \leq \theta_1 <\pi/2 <\theta_2 \leq \pi$ or $0 \leq \theta_1 <\theta_2 \leq \pi/2$.

\begin{figure}[!ht]
\centering
\includegraphics[width =0.48\columnwidth]{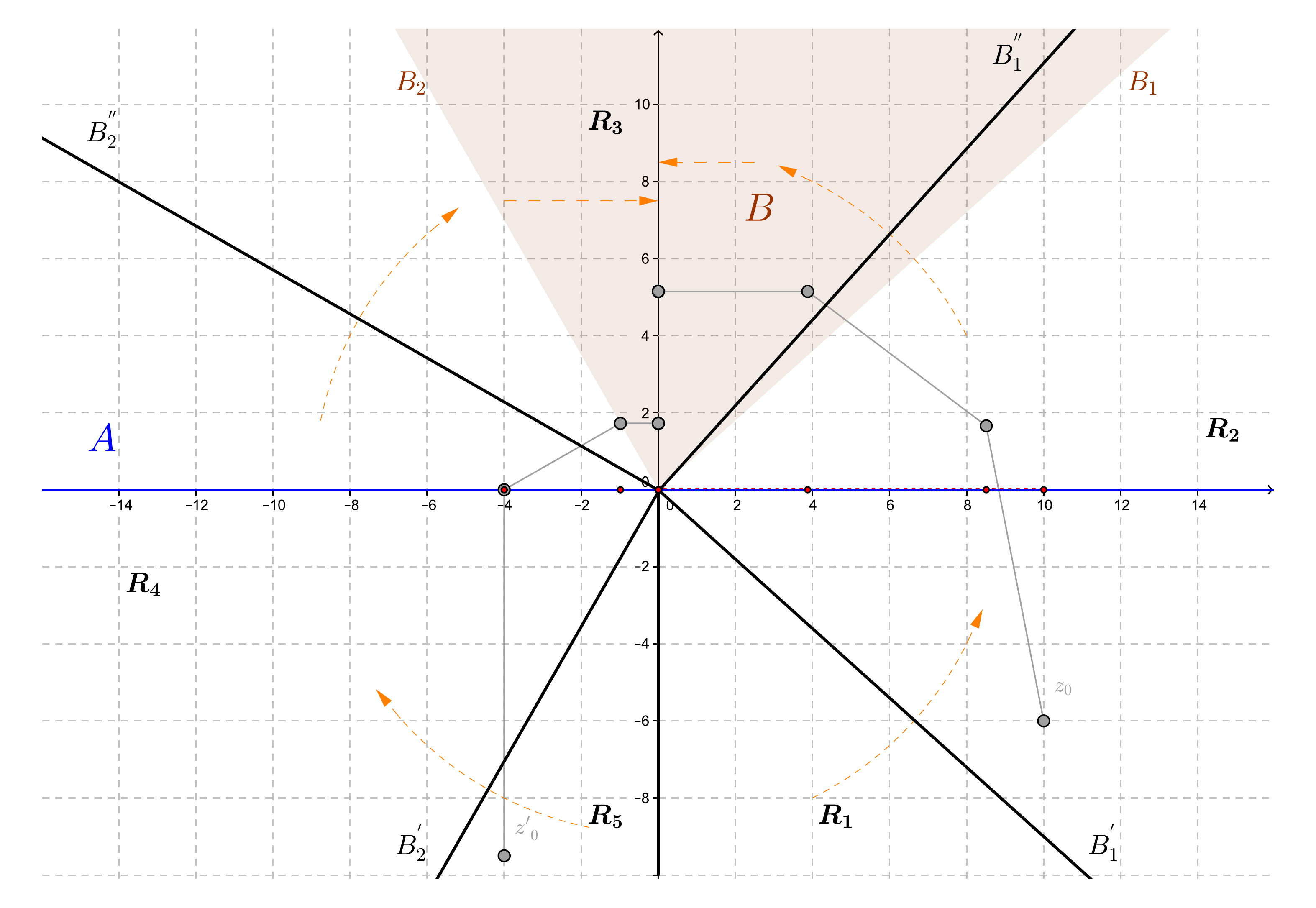}
\includegraphics[width =0.48\columnwidth]{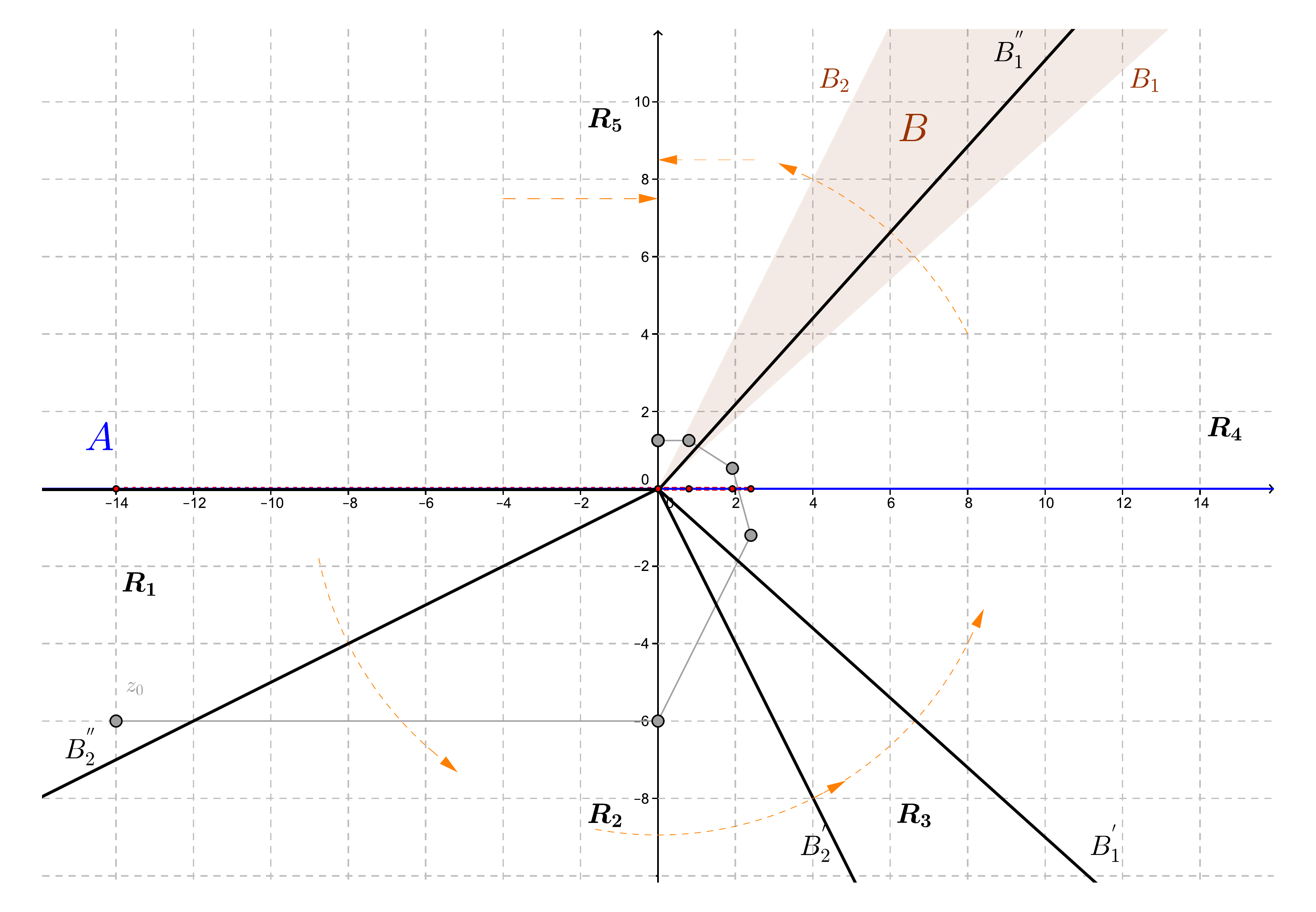}
\caption{The DRA for the case of a line and a cone in the Euclidean plane}
\label{fg:cone}
\end{figure}

\emph{Case 1:} $0 \leq \theta_1 <\pi/2 <\theta_2 \leq \pi$. 
As shown in the left image in Figure~\ref{fg:ray}, 
we study the behavior of the iterations in regions
\begin{subequations}
\begin{align}
R_1 &=\cone(\{0\}\times \RM\cup B_1') =R_A(B)\cap (\RP\times \RR), \\ 
R_2 &=\cone(B_1'\cup B_1'') \subseteq H_1'\smallsetminus R_A(B), \\
R_3 &=\cone(B_1''\cup B_2''), \\
R_4 &=\cone(B_2''\cup B_2') \subseteq H_2'\smallsetminus R_A(B), \\
R_5 &=\cone(B_2'\cup \{0\}\times \RM) =R_A(B)\cap (\RM\times \RR).
\end{align}
\end{subequations}
Set $x_0 :=(\alpha_0, \beta_0)$.

\emph{Case 1.1:} $x_0 \in R_1\cup R_5$. 
Then $P_Ax_0 =(\alpha_0, 0)$, and $R_Ax_0 =(\alpha_0, -\beta_0) \in B =R_A(R_1\cup R_5)$, so
\begin{equation}
x_1 =(\Id -P_A +P_BR_A)x_0 =(\alpha_0, \beta_0) -(\alpha_0, 0) +(\alpha_0, -\beta_0) =(\alpha_0, 0) \in R_2\cup R_4.
\end{equation} 

\emph{Case 1.2:} $x_0 \in R_2$. 
Then $x_0 \in H_1'\smallsetminus R_A(B)$, and $R_Ax_0 \in H_1\smallsetminus B$.
We also see that $R_Ax_0$ belongs to the halfspace with boundary $\lspan B_1$ and not containing $B_2$.
Thus, $P_BR_Ax_0 =P_{B_1}R_Ax_0$, and
\begin{equation}
x_1 =T_{A,B}x_0 =T_{A,B_1}x_0.
\end{equation} 
Using \cref{l:raystep}, this implies
\begin{equation}
x_0, \dots, x_{n-1} \in R_2 \quad\Rightarrow\quad \angle(x_0, x_n) =n\theta_1.
\end{equation}
Therefore, as in the proof of \cref{l:ray}, 
there exists $n_0 \in \NN$, $n_0 \leq \lfloor \pi/(2\theta_1) \rfloor +1$ such that $x_{n_0} \in R_3$.

\emph{Case 1.3:} $x_0 \in R_4$. 
By an argument similar to the above, 
we have $x_{n_0} \in R_3$ for some $n_0 \in \NN$, $n_0 \leq \lfloor \pi/(2\pi -2\theta_2) \rfloor +1$. 

\emph{Case 1.4:} $x_0 =(\alpha_0, \beta_0) \in R_3$. 
Then $\beta_0 \geq 0$ and $R_Ax_0 \not\in H_1\cup H_2$ since $R_3 \not\subseteq H_1'\cup H_2'$.
Therefore, $P_BR_Ax_0 =(0, 0)$, and 
\begin{equation}
x_1 =(\alpha_0, \beta_0) -(\alpha_0, 0) +(0, 0) =(0, \beta_0) \in \{0\}\times \RP \subseteq \Fix T_{A,B}.
\end{equation}
Hence, in all cases, there exists $n_1 \in \NN$ such that 
\begin{equation}
n_1 \leq N :=\max\left\{ \left\lfloor \frac{\pi}{2\theta_1} \right\rfloor, \left\lfloor \frac{\pi}{2(\pi -\theta_2)} \right\rfloor \right\} +3,
\end{equation}
and $x_{n_1} \in \Fix T_{A,B}$. 
This shows that $x_n \to x_{n_1} \in \Fix T_{A,B}$ in at most $N$ iterations.

\emph{Case 2:} $0 \leq \theta_1 <\theta_2 \leq \pi/2$. 
Partitioning
\begin{subequations}
\begin{align}
R_1 &=\cone(B_1''\cup B_2'')\cap (\RR\times \RM), \\ 
R_2 &=\cone(B_2''\cup B_2') \subseteq H_2'\smallsetminus R_A(B), \\
R_3 &=\cone(B_2'\cup B_1') =R_A(B), \\
R_4 &=\cone(B_1'\cup B_1'') \subseteq H_1'\smallsetminus R_A(B), \\
R_5 &=\cone(B_1''\cup B_2'')\cap (\RR\times \RP) 
\end{align}
\end{subequations} 
(see the right image in Figure~\ref{fg:cone}) and arguing as in the above case, 
we obtain that $x_n \to x \in \Fix T_{A,B}$ 
in at most $N$ iterations, where
\begin{equation}
N :=\left\lfloor \frac{\pi}{2\theta_1} \right\rfloor +\left\lfloor \frac{\pi}{2\theta_2} \right\rfloor +5.
\end{equation}
The proof is complete.
\end{proof}

\begin{remark}
\label{r:cone}
By the same argument, \cref{l:cone} also remains true when $\theta_1, \theta_2 \in [-\frac{\pi}{2}, \frac{\pi}{2}]$.
\end{remark}

\begin{theorem}
\label{t:poly2D}
Let $A$ be either a line or a halfplane, and $B$ be a closed convex set in the Euclidean plane $\RR^2$. 
Suppose that $A\cap B \neq\varnothing$, and that $B$ is locally identical with some polyhedral set around $A\cap \bd B$.
Then every DRA sequence $(x_n)_\nnn$ with respect to $(A, B)$ 
converges \emph{finitely} to a point $x \in \Fix T_{A,B}$ with $P_Ax \in A\cap B$.
\end{theorem}
\begin{proof}
Using \cref{t:finite}, it suffices to prove for the case where $B$ is a polyhedral set in $\RR^2$ satisfying $A\cap B \neq\varnothing$.
Then $B =\bigcap_{j \in J} B_j$ is a finite intersection of halfplanes $B_j$. 
Now by \cref{f:cvg}\ref{f:cvg_FixT}, $x_n \to x \in \Fix T_{A,B}$ with $P_Ax \in A\cap B =A\cap (\bigcap_{j \in J} B_j)$.

\emph{Case 1:} $P_Ax$ is not a vertex of $B$.  
Then there exists $j \in J$ such that $B$ and $B_j$ are locally identical around $P_Ax$.
Applying \cref{l:expand} for $A' =A$ and $B' =B_j$, we have 
\begin{equation}
(\exists n_0 \in \NN)(\forall n \geq n_0)\quad x_{n+1} =T_{A,B_j}x_n.
\end{equation}
Since $A$ is either a line or a halfplane, and $B_j$ is a halfplane in $\RR^2$, 
\cref{t:H-H} implies that $x_n \to x$ finitely.

\emph{Case 2:} $P_Ax$ is a vertex of $B$.
Noting that there are exactly two of halfplanes $B_j$ through each vertex of $B$,  
it can also represent $B =\bigcap_{j \in J} C_j$, where each $C_j$ is a closed convex cone in $\RR^2$.
We then find $j \in J$ such that $B$ and $C_j$ are locally identical around $P_Ax$.
By using again \cref{l:expand}, 
\begin{equation}
(\exists n_0 \in \NN)(\forall n \geq n_0)\quad x_{n+1} =T_{A,C_j}x_n.
\end{equation}
Here $A$ is either a line or a halfplane through vertex $P_Ax$ of the cone $C_j$.
Now apply \cref{l:cone} and \cref{r:cone}.
\end{proof}

\begin{example}
\label{ex:poly3D}
Suppose that $X =\RR^3$, that $A =\menge{x \in X}{Lx =a}$, and that $B =\RP^3$, 
where 
\begin{equation}
L =\begin{bmatrix}
1 & 1 & 0 \\
1 & 0 & 1 
\end{bmatrix} 
\quad\text{and}\quad
a =\begin{bmatrix} 1 \\ 0 \end{bmatrix}.
\end{equation}
Then for starting point $x_0 =(1/3, 2/3, 1/3) \in X$, 
the DRA sequence $(x_n)_\nnn$ with respect to $(A, B)$
converges $x_\infty =(1/3, 1, 1/3)$ with $P_Ax_\infty =(0, 1, 0) \in A\cap B$, 
but this convergence is not finite.
\end{example}
\begin{proof}
It is easy to see that $A =\menge{(-\lambda, \lambda+1, \lambda)}{\lambda \in \RR}$, and so 
\begin{equation}
A\cap B =\{(0, 1, 0)\}.
\end{equation}
Let $x =(\alpha, \beta, \gamma) \in X$.
Noting that the Moore--Penrose inverse of $L$ is given by
\begin{equation}
L^\dagger =\frac{1}{3}\begin{bmatrix}
1 & 1 \\ 
2 & -1 \\ 
-1 & 2 
\end{bmatrix},
\end{equation}
we learn from \cref{ex:proj_affine} that 
$P_Ax =x -L^\dagger(Lx -a)$, and so
\begin{equation}
\label{e:RAx}
R_Ax =2P_Ax -x =x -2L^\dagger(Lx -a) 
=\frac{1}{3}\left( \begin{bmatrix}
-1 & -2 & -2 \\ 
-2 & -1 & 2 \\ 
-2 & 2 & -1 
\end{bmatrix}x
+\begin{bmatrix}
2 \\
4 \\
-2 
\end{bmatrix} \right).
\end{equation}
By, e.g., \cite[Example~6.28]{BC11},
$P_Bx =(\max\{\alpha, 0\}, \max\{\beta, 0\}, \max\{\gamma, 0\})$,
and thus
\begin{equation}
\label{e:RBx}
R_Bx =(|\alpha|, |\beta|, |\gamma|).
\end{equation}
Setting $x_+ :=(\alpha_+, \beta_+, \gamma_+) =T_{A,B}x$,
we claim that if 
\begin{subequations}
\label{e:initial}
\begin{align}
\frac{2}{3} &\leq \alpha +\gamma, \\
-\frac{2}{3} &\leq \alpha -\gamma \leq \frac{2}{3}, \\
\frac{2}{3} &\leq \beta \leq \frac{4}{3}, 
\end{align}
\end{subequations} 
then $x_+ =\frac{1}{3}(Mx +b)$, where
\begin{equation}
M :=\begin{bmatrix}
2 & 1 & 1 \\ 
-1 & 1 & 1 \\ 
1 & -1 & 2 
\end{bmatrix}
\quad\text{and}\quad
b :=\begin{bmatrix}
-1 \\ 
2 \\ 
1 
\end{bmatrix},
\end{equation}
and \eqref{e:initial} also holds for $\alpha_+, \beta_+$ and $\gamma_+$.
Indeed, recall that 
\begin{equation}
R_Ax =\frac{1}{3}(-\alpha -2\beta -2\gamma +2, -2\alpha -\beta +2\gamma +4, -2\alpha +2\beta -\gamma -2).
\end{equation}
It follows from \eqref{e:initial} that $\alpha \geq 0$, $\gamma \geq 0$, and
\begin{subequations}
\begin{align}
&-\alpha -2\beta -2\gamma +2 \leq -(\alpha +\gamma) -2\beta +2 \leq -\frac{2}{3} -2\cdot \frac{2}{3} +2 =0, \\
&-2\alpha -\beta +2\gamma +4 =-2(\alpha -\gamma) -\beta +4 \geq -2\cdot \frac{2}{3} -\frac{4}{3} +4 =\frac{4}{3} >0, \\
&-2\alpha +2\beta -\gamma -2 \leq -(\alpha +\gamma) +2\beta -2 \leq -\frac{2}{3} +2\cdot \frac{4}{3} -2 =0.
\end{align}
\end{subequations}
By \eqref{e:RBx} and a direct computation,
\begin{equation}
x_+ =\frac{1}{2}(x +R_BR_Ax) =\frac{1}{3}(Mx +b),
\end{equation}
which means
\begin{subequations}
\begin{align}
\alpha_+ &=\frac{1}{3}(2\alpha +\beta +\gamma -1), \\ 
\beta_+ &=\frac{1}{3}(-\alpha +\beta +\gamma +2), \\ 
\gamma_+ &=\frac{1}{3}(\alpha -\beta +2\gamma +1).
\end{align}
\end{subequations}
Using again \eqref{e:initial} we get
\begin{subequations}
\begin{align}
\alpha_+ +\gamma_+ &=\alpha +\gamma \geq \frac{2}{3}, \\
-\frac{2}{3} <-\frac{4}{9} \leq \alpha_+ -\gamma_+ &=\frac{1}{3}((\alpha -\gamma) +2\beta -2) \leq \frac{4}{9} <\frac{2}{3}, \\
\frac{2}{3} \leq \beta_+ &=\frac{1}{3}(-(\alpha -\gamma) +\beta +2) \leq \frac{4}{3},
\end{align}
\end{subequations}
as claimed. 
Now let $x_0 =(1/3, 2/3, 1/3)$, the above claim implies that
\begin{equation}
(\forall\nnn)\quad x_{n+1} =T_{A,B}x_n =\frac{1}{3}(Mx_n +b).
\end{equation}
A direct argument yields
\begin{equation}
(\forall\nnn)\quad x_{n+3} =\frac{5}{3}x_{n+2} -x_{n+1} +\frac{1}{3}x_n,
\end{equation}
and then
\begin{equation}
x_n =\left( \frac{1}{3} -\frac{\sqrt{2}}{2}\frac{\sin(n\arctan\sqrt{2})}{3^{\frac{n}{2}+1}}, 
1 -\frac{\cos(n\arctan\sqrt{2})}{3^{\frac{n}{2}+1}}, 
\frac{1}{3} +\frac{\sqrt{2}}{2}\frac{\sin(n\arctan\sqrt{2})}{3^{\frac{n}{2}+1}} \right).
\end{equation}
Therefore, $x_n \to x_\infty =(1/3, 1, 1/3)$ linearly with rate $1/\sqrt{3}$, but not finitely.
\end{proof}

\section{Open problems}
\label{s:open}
We conclude with a list of specific open problems.
\begin{itemize}
\sepp 
\item[\textbf{P1}] 
Do the conclusions of \cref{f:P-E} and \cref{t:H-E} hold when $A$ is any hyperplane or halfspace?
\item[\textbf{P2}] 
Does \cref{p:H-B} remain true if $A$ is an affine subspace or a polyhedron?
\item[\textbf{P3}] 
Does \cref{c:H-Bs} remain true without assumption on the uniqueness?
\item[\textbf{P4}] 
Do \cref{c:P-B} and \cref{t:poly2D} remain true in $\RR^n$ with $n >2$?
\item[\textbf{P5}] 
Does \cref{f:A-P} remain true if we replace ``affine subspace'' by ``halfspace''?
\item[\textbf{P6}] 
What can be said about convergence of the DRA for two polyhedrons or for two balls?
\end{itemize}

\subsection*{Acknowledgments}
HHB was partially supported by the Natural Sciences and
Engineering Research Council of Canada 
and by the Canada Research Chair Program.
MND was partially supported by an NSERC accelerator grant of HHB.

\end{document}